\documentclass[journal]{IEEEtran}
\ifCLASSINFOpdf
   \usepackage[pdftex]{graphicx}
\else
\fi

\usepackage{amsmath,amsfonts,amssymb,amsbsy}
\usepackage{subcaption}
\usepackage{tikz}
\usepackage[utf8]{inputenc}
\usepackage{upgreek}
\usepackage{enumitem}
\usepackage{mathrsfs}
\usepackage{nameref}
\usepackage{amssymb}
\usepackage{multirow}
\usepackage{textcomp}
\usepackage{graphicx}
\usepackage{mathtools}
\usepackage[linktocpage]{hyperref}
\usepackage{algorithm}
\usepackage{algpseudocode}
\usepackage{xcolor}
\usepackage{subcaption}
\usepackage{placeins}
\usepackage{hyperref}
\usepackage{caption}
\usepackage{amsthm}
\usepackage{thmtools,thm-restate}
\usepackage{csquotes}
\tikzset{>=latex}

\usepackage{color}
\algtext*{EndWhile}
\algtext*{EndIf}
\algtext*{EndFor}

\newcommand{\supp}[1]{{\mathbf{supp}\left(#1\right)}}
 \hypersetup{
     colorlinks=true,
     linkcolor=blue,
     filecolor=blue,
     citecolor = black,      
     urlcolor=black,
     }

\newcommand{\mcI}{\mathcal{I}}
\newcommand{\mcU}{\mathcal{U}}
\newcommand{\mcC}{\mathcal{C}}
\newcommand{\mcN}{\mathcal{N}}

\newcommand{\tV}{\widetilde{V}}

\newcommand{\tbfv}{{\tilde{\mathbf{v}}}}

\newcommand{\bigO}[1]{\mathcal{O}\left(#1\right)}
\newcommand{\conv}[1]{\mathbf{conv}\left(#1\right)}
\newcommand{\dist}[1]{\mathbf{dist}\left(#1\right)}
\newcommand{\autocorr}[1]{\textbf{\textup{autocorr}}\left(#1\right)}

\newcommand{\mfF}{\mathfrak{F}}

\newcommand{\mfD}{\mathfrak{d}}
\newcommand{\bfx}{\mathbf{x}}

\newcommand{\bfp}{\mathbf{p}}
\newcommand{\bfl}{{\boldsymbol\ell}}
\newcommand{\bfo}{\mathbf{0}}
\newcommand{\bfe}{\mathbf{e}}
\newcommand{\bfc}{\mathbf{c}}
\newcommand{\bfa}{\mathbf{a}}

\newcommand{\bfg}{\mathbf{g}}

\newcommand{\bfj}{\mathbf{j}}

\newcommand{\bfy}{\mathbf{y}}
\newcommand{\bfz}{{\mathbf{z}}}

\newcommand{\tZnd}{\left(\mathbb{Z}_n^d\right)_{\leq 2}}
\newcommand{\bfF}{\mathbf{F}}

\newcommand{\bfv}{{\mathbf{v}}}
\newcommand{\bfw}{{\mathbf{w}}}
\newcommand{\bfu}{{\mathbf{u}}}

\newcommand{\prob}{\mathbb{P}}

\newcommand{\inner}[2]{\langle #1, #2 \rangle}
\newcommand{\diff}[1]{\mathbf{diff}(#1)}
\newcommand{\Vertex}[1]{\mathbf{vert}\left(#1\right)}
\newcommand\smalldots{\makebox[0.75em][c]{.\hfil.\hfil.}}
\DeclareMathOperator*{\argmin}{arg\,min}
\DeclareMathOperator*{\argmax}{arg\,max}

\newtheorem{theorem}{Theorem}

\newtheorem{definition}{Definition}
\newtheorem{corollary}[theorem]{Corollary}

\begin{document}

\title{Support Recovery for Sparse Multidimensional Phase Retrieval}

\author{Alexei Novikov, Stephen White
\thanks{A. Novikov is with the Department of Mathematics, Penn State University, University Park, PA 16802, USA e-mail: novikov@psu.edu} 
\thanks{S. White is with the Department of Mathematics, Penn State University, University Park, PA 16802, USA e-mail: sew347@psu.edu}
\thanks{This paper has supplementary downloadable material available under ``media'' at http://ieeexplore.ieee.org, provided by the authors, which includes detailed proofs of our results, additional numerical simulations, and a table of reference for our many notations.}
}

\markboth{IEEE Trans. on Signal Processing}%
{Novikov and White: Support Recovery for Sparse Multidimensional Phase Retrieval}
\maketitle  
\begin{abstract}
We consider the \textit{phase retrieval} problem of recovering a sparse signal $\bfx$ in $\mathbb{R}^d$ from intensity-only measurements in dimension $d \geq 2$. Phase retrieval can be equivalently formulated as the problem of recovering a signal from its autocorrelation, which is in turn directly related to the combinatorial problem of recovering a set from its pairwise differences. In one spatial dimension, this problem is well studied and known as the \textit{turnpike problem}. In this work, we present MISTR (Multidimensional Intersection Sparse supporT Recovery), an algorithm which exploits this formulation to recover the support of a multidimensional signal from magnitude-only measurements. MISTR takes advantage of the structure of multiple dimensions to provably achieve the same accuracy as the best one-dimensional algorithms in dramatically less time. We prove theoretically that MISTR correctly recovers the support of signals distributed as a Gaussian point process with high probability as long as sparsity is at most $\bigO{n^{d\theta}}$ for any $\theta < 1/2$, where $n^d$ represents pixel size in a fixed image window. In the case that magnitude measurements are corrupted by noise, we provide a thresholding scheme with theoretical guarantees for sparsity at most $\bigO{n^{d\theta}}$ for $\theta < 1/4$ that obviates the need for MISTR to explicitly handle noisy autocorrelation data. Detailed and reproducible numerical experiments demonstrate the effectiveness of our algorithm, showing that in practice MISTR enjoys time complexity which is nearly linear in the size of the input.

\end{abstract}

{\section{Introduction}}
In many imaging setups, technological or physical constraints prevent sensors from detecting phases of a signal, requiring practitioners to reconstruct signals of interest from only their magnitudes. The \textit{phase retrieval} problem of recovering a signal from magnitude-only measurements is consequently of significant importance to a number of fields, including biological imaging \cite{WaXuLaSoLi:2013}, X-ray crystallography \cite{Mi:1990}, and many others \cite{JaElHa:2015}.

For fixed integer dimension $d \geq 1$, let $\bfx$ be a $d$-dimensional signal with values in $\mathbb{C}$. If $\bfy$ is the Fourier transform of $\bfx$ and $\bfF$ is the $d$-dimensional Fourier transform operator, then phase retrieval can be expressed as:
\begin{equation}\label{phase_retrieval}
    \text{Find } \bfx \text{ subject to } |\bfF\bfx|^2 = |\bfy|^2
\end{equation}

Given its long history of applications, phase retrieval has been studied for many years. The most popular methods are based on the alternating projection algorithms originally devised by Gerchberg and Saxton \cite{GeSa:1972}, which were further refined by Fienup \cite{Fi:1982}. These methods iteratively project data between the Fourier and spatial domains, updating their respective coefficients at each step to adhere to known constraints. However, this technique involves projection onto a non-convex set and so convergence to a global minimum is not guaranteed \cite{BaCoLu:2002}; consequently, significant a priori information about the signal is needed to guarantee convergence to the correct solution.

The lost phase information makes phase retrieval a difficult problem to solve in full generality, so researchers often assume additional information to make the problem more tractable. In many applications, it is natural to assume the signal $\bfx$ is sparse, i.e. $\bfx$ has only a small number of nonzero coefficients when discretized in some basis. This has informed greedy search methods with sparsity constraints \cite{ShBeEl:2014}, but these methods suffer from scalability issues when $n$ grows large \cite{JaOyHa:2017}. An alternative approach reframes phase retrieval as a sparse matrix retrieval problem, which can be solved by semi-definite programming after an appropriate convex relaxation \cite{ChMoPa:2010} \cite{CaElStVo:2013}. When the image is sufficiently sparse, these methods guarantee recovery under minimal constraints on the image. However, solving a matrix rather than a vector problem makes the resulting solutions computationally expensive.

To avoid this quadratic scaling, some researchers have proposed non-convex minimization approaches. These methods make an initial guess using a spectral method, then perform gradient descent on the signal domain \cite{CaLiSo:2015}, \cite{WaGiGeCh:2018}, \cite{BeElBo:2018}. If the initial guess is sufficiently accurate, the iterates provably converge to the true solution (up to a global phase) at a geometric rate; however, these methods perform poorly if the initial guess is inaccurate. Other non-convex approaches adapt alternating projection algorithms by adding an initialization step and employing a resampling approach \cite{PrPrSu:2015}.

\subsection{The combinatorial approach}
Jaganathan et al. \cite{JaOyHa:2017} suggested a different approach: they exploited the fact that the problem of recovering the support of $\bfx$ could be reformulated as the following combinatorial problem. When the domain of $\bfx$ is padded with extra zeros, one can prove that the autocorrelation $\bfa$ of $\bfx$ is the inverse Fourier transform of the squared Fourier magnitudes, $\bfa = \bfF^{-1}(|\bfF\bfx|^2)$. In this setup, the phase retrieval problem can be equivalently expressed as the problem of recovering a signal $\bfx$ from its known autocorrelation $\bfa$: 
\begin{equation}
\text{find } \bfx \text{ subject to } \bfa = \bfF^{-1}|\bfF\bfx|^2
\end{equation}
By properties of the autocorrelation, it is known that as long as signal values follow a continuous distribution, finding the support of the signal $\bfx$ is equivalent to the following combinatorial problem \cite{RaChLuVe:2013} \cite{JaOyHa:2017}:
\begin{definition}[Support Recovery: Combinatorial Form]
\begin{equation}\label{turnpike}
    \text{Find } V \text{ subject to } W = \{\bfv - \bfv' : \bfv,\bfv' \in V\} \subseteq \mathbb{R}^d
\end{equation}
\end{definition}
We refer to this problem simply as the ``combinatorial problem." In the phase retrieval context, $V$ is the support of $\bfx$ and $W$ the support of $\bfa$, both subsets of $\mathbb{R}^d$. We denote $\supp{\bfx}$ the support of $\bfx$, and we call $\diff{V} := \{\bfv - \bfv' : \bfv,\bfv' \in V\}$ the \textit{difference set} of $V$; we will often denote $\diff{V}$ by $W$. Jaganathan et al. exploited this formulation with two-stage sparse phase retrieval (TSPR), which first solves (\ref{turnpike}) when $d = 1$ using an $\bigO{k^4}$ algorithm, where $k$ is the number of elements in $\supp{\bfx}$. Signal values are then recovered by a convex relaxation approach.

As data in imaging applications can often be represented as two- or three-dimensional, improved algorithms for phase retrieval in this setting are of significant practical interest. It is known that phase retrieval becomes easier to solve in multiple dimensions, where the problem enjoys better uniqueness properties \cite{BrSo:1979} \cite{Ha:1982}. Yet the standard approach to solve the combinatorial problem in higher dimensions (see e.g.~\cite{LeSkSm:2003}), is to leverage existing one-dimensional algorithms by lifting them essentially unchanged to higher dimensions, or by projecting multidimensional data onto a one-dimensional subspace. Such an approach neglects any distinct properties of geometry in higher dimensions. In particular, in dimension greater than one the same data can be ordered in several geometrically-compatible ways by inner product with different directions. We exploit this fact to achieve better algorithmic performance, improving upon TSPR with an algorithm for dimension $d \geq 2$ that, with high probability, recovers $\supp{\bfx}$ from its autocorrelation in $\bigO{k^2\log{k}}$ time. This represents a runtime improvement of several orders of magnitude relative to what has previously been achieved in a phase retrieval setting.

\subsection{Related work for the combinatorial problem}
Like phase retrieval, the combinatorial problem is fundamentally ill-posed: given any set $V$ the sets $V+\bfa$ and $-V$ will also produce the same difference set $W$. These are commonly referred to as \textit{trivial ambiguities} and are unavoidable. Thus in solving the combinatorial problem (\ref{turnpike}), we aim to recover an equivalent solution to $V$, which differs only by a constant shift and possible multiplication by $-1$. Even then, solutions to the combinatorial problem are not necessarily unique, though some partial uniqueness results are known \cite{RaChLuVe:2013}. 

In \cite{LeWe:1988}, Lemke and Werman developed a factorization algorithm to recover subsets $V \subseteq \{0,1,\ldots, n-1\}$ from pairwise differences in $\bigO{k^{w_{\max}}}$ time, where $w_{\max}$ is the largest difference in $W$ and $k = |V|$; however, this becomes impractical as $|V|$ grows large. Lemke et al. later proposed a backtracking algorithm \cite{LeSkSm:2003} which has an exponential worst-case \cite{Zh:1994} but runs faster in practice. Both of the above algorithms, however, assume multiplicity information---that is, it is known a priori how many different pairs create each particular difference. This multiplicity information is not available in the phase retrieval formulation.

Without the assumption of multiplicity information, in \cite{JaOyHa:2017}, Jaganathan et al. showed that polynomial recovery times are achievable when $V$ is sparse. The primary tool in their approach is a $\bigO{k^2\log(k)}$ \textit{intersection step} that deletes many differences in $W$ not in the support of $V$ without deleting any support elements.

\subsection{Our contribution and organization of paper}

When this intersection step is not sufficient to find a solution, the TSPR algorithm relies on a $\bigO{k^4}$ \textit{graph step} which can result in false negatives. In our present work, we show that in multiple dimensions one can solve the combinatorial problem by combining information from multiple compatible intersection steps, without needing this $\bigO{k^4}$ operation. Specifically, we present MISTR (Multidimensional Intersection Sparse supporT Recovery), an algorithm which solves (\ref{turnpike}) when $d \geq 2$ which recovers the support of a signal $\bfx$ in $\bigO{k^2 \log(k)}$ time with high probability. In the sparse case under consideration, $\diff{V}$ typically has on the order of $k^2$ elements. This means our algorithm runs in time nearly linear in the size of the input.

We provide theoretical guarantees and numerical results for a probabilistic model of the unknown set $V$. Specifically, we prove accuracy guarantees for sparsity up to $\bigO{n^{d\theta}}$ where $\theta \in (0,1/2)$ and $n$ is a resolution parameter. In this formulation, $\theta$ determines how sparse the data is and accordingly represents the primary criterion for difficulty. In particular, as $\theta$ increases, the number of repeated differences in $W$ grows; this was identified as a key indicator of the difficulty of the combinatorial problem in \cite{HuDo:2020}. In this light, our primary contribution is a method for much faster support recovery when $\theta \in [1/4,1/2)$ which runs in $\bigO{k^2\log{k}}$ time while the leading alternative requires $\bigO{k^4}$ for this case. This low time complexity makes MISTR particularly appealing for large-scale phase retrieval problems provided the data is still sufficiently sparse. 

Lastly, we provide theoretical and numerical results when data are corrupted by noise, showing that when combined with a simple denoising method, MISTR recovers $\bigO{n^{d\theta}}$-sparse signals for $\theta \in (0,1/4)$ with high probability without requiring any modifications to the MISTR algorithm.

The outline of the paper is as follows. In section \ref{3_algorithm}, we introduce the MISTR algorithm for efficiently solving (\ref{turnpike}) when $d \geq 2$. In section \ref{4_results}, we prove accuracy guarantees for this algorithm that guarantee recovery as long as the set $V$ is sufficiently sparse, and comment on related guarantees in the event that the data is corrupted with additive noise. Lastly, in section \ref{5_numerical}, we provide a number of reproducible numerical simulations which verify the effectiveness of MISTR in practice. All code used in the numerical experiments can be found at \url{https://github.com/sew347/Sparse-Multidimensional-PR}.

\section{Algorithm}\label{3_algorithm}
\subsection{Notation and Setting}
Throughout, we follow the convention that boldface lower-case letters represent vectors, and non-boldface lower-case and greek letters represent scalars. Let $\mathbb{Z}^d$ be the grid of integer-valued vectors in $\mathbb{R}^d$ for dimension $d \geq 2$, and let $\mathbb{Z}_n^d = \left(\frac{1}{n}\mathbb{Z}\right)^d$. $1/n$ can be understood as resolution and $n^d$ as number of points in a discrete image window, up to a constant. We note that, while it is conceptually helpful to envision $n$ as dividing the box $[0,1]^d$ into $n^d$ gridpoints, our construction does not require that $n$ is an integer, and so may not subdivide a box with integer side lengths exactly. We will write $|A|$ for the number of points in a finite set $A$. Throughout the paper we use $\ln(x)$ to refer to the natural logarithm of $x$, while $\log(x)$ refers to the logarithm of base 2.

Let $V$ denote our unknown support set $V \subseteq \mathbb{Z}_n^d$. As we are approaching this problem from the perspective of imaging, we will refer to points in $V$ \textit{scatterers}. Let $k$ be the (unknown) number of scatterers in $V$. We write $W = \diff{V}$ and denote $\kappa = |W|$. Though the exact number of scatterers is unknown a priori, we compute the minimum number of scatterers required to create $\kappa$ differences: we call this number $k_{\min}(\kappa)$. Since $W$ has at most $k(k-1)+1$ elements if all differences are distinct, $k \geq k_{\min}(\kappa) = \left\lceil \frac{1}{2}\left(1+\sqrt{4\kappa - 3}\right) \right\rceil$. When $\kappa$ is clear from context, we will often drop the dependence on $\kappa$ and write $k_{\min} := k_{\min}(\kappa)$.

We model sparsity $s$ as a power of $n^d$: that is, $s := n^{d\theta}$ for a fixed sparsity parameter $\theta$; it is straightforward to generalize our results to the case $s = \bigO{n^{d\theta}}$. Thoughout this work, $s$ represents a sparsity-determining parameter in a probabilistic model, while $k$ represents the actual number of points in a realization of such a model.

We will often order the set $V$ by its elements' inner product with different random vectors $\bfz^i$. For two vectors $\bfy$ and $\bfv$, we say $\bfy < \bfv$ with respect to $\bfz^i$ if and only if $\inner{\bfy}{\bfz^i} < \inner{\bfv}{\bfz^i}$. We index the set $V$ according to $\bfz^i$ using the following notation: if $\bfv_j^i, \bfv_m^i \in V$, then $j < m \iff \bfv_j^i < \bfv_m^i$ with respect to $\bfz^i$. 

In our supplementary media, we have included a table which summarizes our notations for readers to reference.

\subsection{The MISTR Algorithm: Introduction}

We are now ready to introduce our algorithm, \textbf{Multidimensional Intersection Sparse supporT Recovery} (MISTR), which solves (\ref{turnpike}) with high probability. As inputs, MISTR takes an integer $\uptau$ and difference set $W$ \textit{without} multiplicity. The algorithm outputs an equivalent solution to $V$ if one is found, or a best guess if no exact solution is found.

First, MISTR generates $\uptau$ random vectors $\bfz^1, \ldots, \bfz^\uptau$ from the sphere $\mathbb{S}^{d-1}$. MISTR consists of a \textit{projection step} followed by an \textit{intersection step} for each $\bfz^i$. After $\uptau$ intersection steps, if no solution was found, MISTR employs a \textit{collaboration search} to find a solution.

For each $i$, MISTR attempts to recover a set $\widetilde{V}^i$ which is equivalent to the true support set $V$. Specifically, for each $i$, the intersection step will return a set $U^i$ which is guaranteed to contain a set $\tV^i$ that is equivalent to $V$, defined by:
\begin{equation}\label{truesoln}
\widetilde{V}^i = \begin{cases} V - \bfv_0^i, & \inner{\bfv_1^i - \bfv_0^i}{\bfz^i} \leq \inner{\bfv_{k-1}^i - \bfv_{k-2}^i}{\bfz^i} \\ \bfv_{k-1}^i - V, & \text{ otherwise} \end{cases}
\end{equation}
Essentially, $\tV^i$ changes coordinates of $V$ by choosing either $\bfv_0^i$ or $\bfv_{k-1}^i$ as the origin of the new coordinate system, based on the condition in (\ref{truesoln}), flipping by $-1$ in the latter case. We will denote elements in $\widetilde{V}^i$ as $\tbfv^i_j$, where the ordering in the subscript is induced by the inner product with $\bfz^i$. By this definition, we always have $\tbfv_0^i = 0$. Further, from (\ref{truesoln}) one can see that any $\tbfv\in \tV^i$ is the difference between some point in $V$ and either $\bfv_0^i$ or $\bfv_{k-1}^i$. It follows that $\tV^i \subseteq W$. Thus our algorithm only needs to remove false positives: vectors $\bfl$ such that $\bfl \in W$ but $\bfl \notin \widetilde{V}^i$. The goal of the intersection step is to efficiently eliminate as many of these false positive vectors as possible without removing any elements of $\tV^i$.

\begin{algorithm}[h]
\textbf{Input:} A set $W$ that is the set of pairwise differences between points in an unknown set $V$, without multiplicities.\\
\textbf{Output:} A set $U$ such that $W \subseteq \diff{U}$ and a boolean $B$ indicating if $W = \diff{U}$
\begin{algorithmic}[1]
\caption{MISTR} \label{MISTR}
\State Select $\uptau$ random vectors $\bfz^1,\ldots,\bfz^\uptau$ uniformly at random from $\mathbb{S}^{d-1}$.
\For{$i = 1:\uptau$}
    \State \textit{Projection Step:} Compute $W^i$ by projecting $W$ onto \par
    \hspace*{8pt}$\bfz^i$, sorting, and selecting differences with positive \par
    \hspace*{8pt}inner product with $\bfz^i$.
    \State Deduce $\tbfv_{1}^i = \bfw_{\kappa-1}^i - \bfw_{\kappa-2}^i$.
    \State \textit{Intersection Step:} $U^i = \{0\} \cup \left[W^i \cap (W^i + \tbfv_{1}^i)\right]$.
\EndFor
\State \textit{Collaboration Search}: Perform a breadth-first search for a solution on a particular (virtual) tree $\mathcal{T}$, using known properties of the sets $\tV^i$ to prune the vast majority of non-solutions.
\end{algorithmic}
\end{algorithm}

\subsection{Projection and intersection steps}

\textit{Projection Step:} The difference set $W$ is ordered according to projection with $\bfz^i$, and $W^i$ is taken to be the subset of vectors in $W$ with nonnegative projections, indexed by the ordering induced by $\bfz^i$. Since $W = -W$, we do not sacrifice any information about $V$ by restricting $W^i$ to this half-space. As noted above, we know $\tbfv_0^i = 0$. We can also infer that $\tbfv_{1}^i = \bfw_{k-1}^i - \bfw_{k-2}^i$ by the following reasoning. It is immediate that the vector with the largest projection must be included in $\tV^i$; that is, $\bfw_{\kappa-1}^i = \tbfv_{k-1}^i$. Further, by definition of $\tV^i$, $\inner{\tbfv_1^i - \tbfv_0^i}{\bfz^i} \leq \inner{\tbfv_{k-1}^i - \tbfv_{k-2}^i}{\bfz^i}$. Thus, $\bfw_{\kappa-2}^i = \tbfv_{k-1}^i - \tbfv_1^i$, so it follows that $\bfw_{\kappa-1}^i - \bfw_{\kappa-2}^i = \tbfv_{k-1}^i - (\tbfv_{k-1}^i-\tbfv_1^i) = \tbfv_{1}^i$.

\textit{Intersection Step:} This step is the same as the intersection step employed by Jaganathan et al. in \cite{JaOyHa:2017}. This step returns:
\[
    U^i = \{0\} \cup \left[W^i \cap (W^i + \tbfv_{1}^i)\right]
\]
Here $W^i + \tbfv_{1}^i = \left\{\bfw + \tbfv_{1}^i : \bfw \in W^i \right\}$. We quote directly the description of this step given in \cite{JaOyHa:2017} (notation has been adapted to our setting):
\begin{displayquote}
This step can be summarized as follows: suppose we know the value of $\tbfv_{m}^i$ for some $m$, then
\[
\left\{\tbfv_{j}^i : m \leq j \leq k-1\right\} \subseteq W^i \cap (W^i+\tbfv_{m}^i)
\]
.... This can be seen as follows: $\tbfv_{j}(i) \in W^i$ by construction for $0 \leq j \leq k-1$. $\tbfv_{m}^i - \tbfv_{j}^i \in W^i$ by construction for $0 \leq j \leq k-1$, which when added by $\tbfv_{m}^i$ gives $\tbfv_{j}^i$ and hence $\tbfv_{j}^i \in W^i \cap (W^i+\tbfv_{m}^i)$ for $m \leq j \leq k-1$.
\end{displayquote}

We will write elements in $U^i$ as $\bfu_j^i$, where the superscript $i$ corresponds to the $i$-th random vector $\bfz^i$ and the subscript $j$ corresponds to the ordering induced by $\bfz^i$. Since we do not know how many elements are in $U^i$ a priori, we refer to the largest element in $U^i$ by this ordering as $\bfu_{\max}^i$. We note that by construction, $\bfu_0^i = \tbfv_0^i = 0$, $\bfu_1^i = \tbfv_1^i$, and $\bfu_{\max}^i = \tbfv_{k-1}^i$.

After $U^i$ is found, if we have $|U^i| = k_{\min}(\kappa)$, we know that it must be the true solution $\widetilde{V}^i$. This is because $|U^i| = k_{\min}(\kappa)$ means $U^i$ contains the minimum number of elements required to produce a difference set with $\kappa$ elements, and since $\widetilde{V}^i \subseteq U^i$ it follows that $\widetilde{V}^i = U^i$. If $|U^i| > k_{\min}(\kappa)$, we repeat the projection and intersection steps for $i+1$ up to $\uptau$. 

One projection step takes $\bigO{k^2\log(k)}$ time to compute and sort $W^i$, and an intersection step takes $\bigO{k^2\log(k)}$ time to compute $W^i \cap \left(W^i + \bfv_0^i\right)$ \cite{DiKo:2011}. Thus performing $\uptau$ intersection steps takes at most $\bigO{\uptau k^2\log(k)}$ time. 

When the number of scatterers is small ($\theta < 1/4$), the projection and intersection steps are usually sufficient to find a solution. Indeed, it is proved in $\cite{JaOyHa:2017}$ that in one dimension a single intersection step can recover $\bigO{n^{1/4}}$ signals for large enough $n$. However, for $\theta \geq 1/4$ this is not sufficient. 
\subsection{Collaboration search}
Though the individual sets $U^i$ are rarely solutions when the number of scatterers grows large, we do know that each $U^i$ contains an equivalent solution $\tV^i$. It thus stands to reason that taking the intersection would improve the chances of removing all false positive vectors. Care is needed, however: while the recovered solution always contains \emph{an} equivalent solution, these solutions may differ by a possible global shift and negative flip per (\ref{truesoln}). Thus we need to reorient the solutions $U^i$ before we can take their intersection in a meaningful way.

From (\ref{truesoln}) we see that each $\tV^i$ differs from $\tV^1$ by at most a constant shift and possible multiplication by $-1$. Put another way, for each $i$ there exists a transformation $\Theta^i$ consisting of offset by a particular vector and possible multiplication by $-1$, such that $\Theta^i(\widetilde{V}^i) = \widetilde{V}^1$; for convenience we take $\Theta^1$ to be the identity. Moreover, since $0 \in \tV^1$ by construction, the only possible candidates for $\Theta^i$ take the form $\Theta^i(\tbfv) = \pm(\tbfv - \tbfv^i_j)$ for some $\tbfv_j^i \in \tV^j$. For example, if $\widetilde{V}^1 = V - \bfv_0^1$ and $\widetilde{V}^2 = \bfv_{k-1}^2 - V$, then $\widetilde{V}^1 = \bfv_{k-1}^2-\bfv_0^1-\tV^2$, in which case $\Theta^2(\bfv) = \bfv_{k-1}^2-\bfv_0^1-\bfv$, noting that $\bfv_{k-1}^2-\bfv_0^1\in \tV^2$.

The goal of the collaboration search is to find the re-oriented intersection
\[
U = \bigcap_{i=1}^\uptau \Theta^i(U^i)
\]
which will always contain $\tV^1$. The difficulty lies in the fact that the sets $\tV^i$ are unknown \textit{a priori}, so we must infer the transformations $\Theta^i$ from the retrieved sets $U^i$. 

There are only a finite number of choices for $\Theta^i$, so in principle we could iterate through them and check if their intersections are solutions. Such an operation would clearly be exponential in the number of elements in $V$. Yet we can derive several conditions which dramatically reduce the size of the feasible set, allowing us to quickly discard almost all choices for $\Theta^i$. To do this, we formulate the problem as a tree-based search algorithm. We will need the following definitions:
\begin{definition}[Search Sequence]\label{search}
We define a \textbf{search sequence} for $\{U^1,\ldots,U^p\}$ as a sequence $\bfj$ of $p$ pairs $(j_i, \omega_i)$ satisfying:
\begin{enumerate}
    \item $j_1 = 0$ and $\omega_1 = 1$
    \item $j_i \in \{0,1,2,\ldots,|U^i|-1\}$, $i \geq 2$
    \item $\omega_i \in \{1,-1\}$, $i \geq 2$.
\end{enumerate}
We call $p$ the \emph{depth} of the sequence.
\end{definition}
Together, all search sequences comprise all possible ways one could choose an index for an element in $U^i$ and a sign in $\{1,-1\}$ for each $i = 1,2,\ldots,p$. We now define the following functions that we will use to describe the collaboration search:
\begin{definition}[Orientation]
Let $\bfj$ be a search sequence for $\{U^1,\ldots,U^p\}$. Then for each $i$, define the $i$-th \textbf{orientation} of $\bfj$ as the function $O_\bfj^i$ defined for $\bfu \in U^i$ given by:
\[
O_\bfj^i(\bfu) = \omega_i(\bfu - \bfu_{j_i}^i)
\]
\end{definition}
Intuitively, the orientation $O_\bfj^i$ is a coordinate shift of $U^i$ so that $\bfu_{j_i}^i$ is the origin, along with a possible flip by a factor of $-1$ depending on $\omega_i$. We now introduce the concept of a \textit{collaboration}:
\begin{definition}[Collaboration]
Given a search sequence $\bfj$, the \textbf{collaboration} with respect to $\bfj$ is the intersection of the orientations of $\bfj$ over each $i = 1,2,\ldots,p$:
\[
\mcC_\bfj = U^1 \cap \bigcap_{i=2}^p \omega_i\left(U^i - \bfu^i_{j_i}\right)  = \bigcap_{i=1}^p \omega_i\left(U^i - \bfu^i_{j_i}\right) = \bigcap_{i=1}^p O_\bfj^i(U^i)
\]
We call $p$ the \emph{depth} of the collaboration. 
\end{definition}

It follows that $U=\bigcap_{i=1}^\uptau \Theta^i(U^i)$ is a collaboration of depth $\uptau$. Denote $\bfj^*$ the search sequence of length $\uptau$ such that $U = \mcC_{\bfj^*}$, and denote by $\bfj^*_p$ the first $p$ terms of $\bfj^*$. Correspondingly, let $U_p^* = \mcC_{\bfj_p^*}=\bigcap_{i=1}^p \omega_i^*\left(U^i - \bfu^i_{j_i^*}\right)$. We note that $O_{\bfj^*}^i = \Theta^i$. (Though such an event is rare, it is possible that $\bfj^*$ is not unique. Nonetheless, any such index will provide the solution $U$. For simplicity, we explain the collaboration search as if this index is unique.)

In this terminology, the naive approach described above to find $\bfj^*$ would be to compute $\mcC_\bfj$ for every possible search sequence $\bfj$ and check whether $\diff{\mcC_\bfj} = W$. While impractical, this idea informs the intuition that we can treat the problem of finding $\bfj^*$ as a tree-based search algorithm. To make this efficient, we leverage known information about $U_p^*$ to prune branches at each depth to dramatically reduce search time. We call this iterative pruning and search process the \textit{collaboration search}.

Formally, we consider the \textit{search tree} $\mathcal{T}$ whose nodes are all search sequences $\bfj$ with depth between $1$ and $\uptau$ with the following parent-child structure. The root of the tree is the node consisting of the single-element search sequence $(0, 1)$. The nodes at each depth $p$ will be exactly all search sequences of depth $p$, where the parent of each node $\bfj_p$ will be the subsequence $\bfj_{p-1}$ consisting of the first $p-1$ elements of $\bfj_p$. The leaves of this tree will be all possible search sequences $\bfj$ of depth $\uptau$. An example of the initial setup of this tree can be found in figure \ref{depth2_1}. The collaboration search is a breadth-first search with pruning on this tree that uses the structure of $U_p^*$ to avoid searching the vast majority of the tree's nodes. We note that, while the collaboration search step is best described in terms of this search tree, it is never necessary to actually construct $\mathcal{T}$, which would take an impractical amount of time and space.

To demonstrate how the collaboration search operates on the tree $\mathcal{T}$, we introduce an explicit example. For simplicity of exposition, we pretend that $k$ is known, i.e. set $k_{\min} = k$ for this example. Let $d = 2, \uptau = 3$ and let $V$ be the set:
\[
V = \scriptsize\left\{\begin{pmatrix} -2 \\ 0\end{pmatrix}, \begin{pmatrix} -1 \\ -1\end{pmatrix}, \begin{pmatrix} -1 \\ 0\end{pmatrix}, \begin{pmatrix} 0 \\ 2\end{pmatrix}, \begin{pmatrix} 1 \\ -1\end{pmatrix}, \begin{pmatrix} 1 \\ 0 \end{pmatrix},\begin{pmatrix} 2 \\ 1\end{pmatrix}\right\}
\]\normalsize
It follows that $k=7$. Set $\bfz^1 = (0.911, 0.413)^T$, $\bfz^2 = (0.974,0.228)^T$, and $\bfz^3 = (0.0266,0.9996)^T$ where $T$ denotes matrix/vector transpose. After the intersection step, we are left with (in order according to projection with $\bfz^1$, $\bfz^2$, $\bfz^3$ respectively):
\scriptsize\[
U^1 = \tiny\left\{\begin{pmatrix} 0 \\ 0\end{pmatrix}, \begin{pmatrix} 1 \\ -2\end{pmatrix}, \begin{pmatrix} 0 \\1\end{pmatrix}, \begin{pmatrix} 1 \\ 0 \end{pmatrix}, \begin{pmatrix} 2 \\ -1\end{pmatrix}, \begin{pmatrix} 2 \\ 0\end{pmatrix},\begin{pmatrix} 3 \\-1\end{pmatrix}, \begin{pmatrix} 2 \\2\end{pmatrix},\begin{pmatrix} 3 \\ 0\end{pmatrix},\begin{pmatrix} 4 \\ 1\end{pmatrix}\right\}
\]\scriptsize
\[
U^2 = \tiny\left\{\begin{pmatrix} 0 \\ 0\end{pmatrix}, \begin{pmatrix} 1 \\ 2\end{pmatrix}, \begin{pmatrix} 1 \\1\end{pmatrix}, \begin{pmatrix} 1 \\ 0 \end{pmatrix}, \begin{pmatrix} 2 \\ 2\end{pmatrix}, \begin{pmatrix} 2 \\ 0\end{pmatrix},\begin{pmatrix} 2 \\-1\end{pmatrix}, \begin{pmatrix} 3 \\3\end{pmatrix},\begin{pmatrix} 3 \\ 1\end{pmatrix},\begin{pmatrix} 4 \\ 1\end{pmatrix}\right\}
\]\scriptsize
\[
U^3 = \tiny\left\{\begin{pmatrix} 0 \\ 0\end{pmatrix}, \begin{pmatrix} -2 \\ 1\end{pmatrix}, \begin{pmatrix} 0 \\1\end{pmatrix}, \begin{pmatrix} 1 \\ 1 \end{pmatrix}, \begin{pmatrix} -1 \\ 2\end{pmatrix}, \begin{pmatrix} 0 \\ 2\end{pmatrix},\begin{pmatrix} 1 \\2\end{pmatrix}, \begin{pmatrix} 2 \\2\end{pmatrix},\begin{pmatrix} -1 \\ 3\end{pmatrix},\begin{pmatrix} 1 \\ 4\end{pmatrix}\right\}
\]\normalsize
As none of these can be an equivalent solution to $V$ (they all have size greater than $k$), the collaboration search is necessary. The initial setup of the search tree $\mathcal{T}$ can be seen in figure \ref{depth2_1}.

\begin{figure*}[t]
\centering
\begin{subfigure}{.5\textwidth}
    \centering
    \includegraphics[width=1\linewidth]{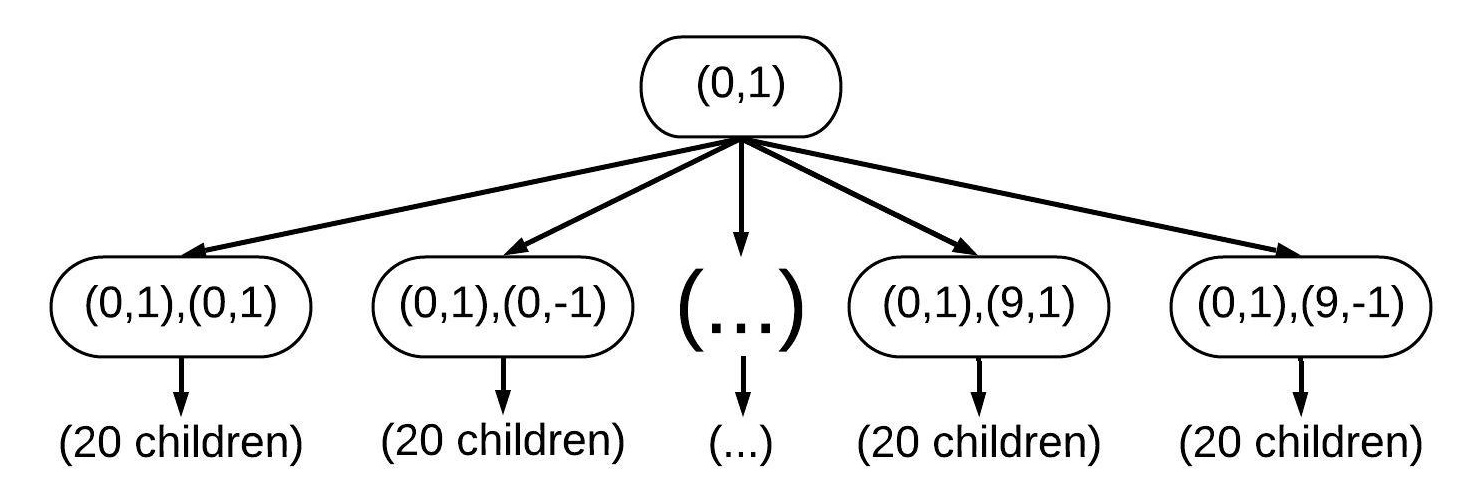}
    \caption{Initial setup of search tree $\mathcal{T}$. Each node at depth $p$ has twice as many children as $U^{p+1}$ has elements.}
    \label{depth2_1}
\end{subfigure}\hfill
\begin{subfigure}{.5\textwidth}
    \centering
    \includegraphics[width=1\linewidth]{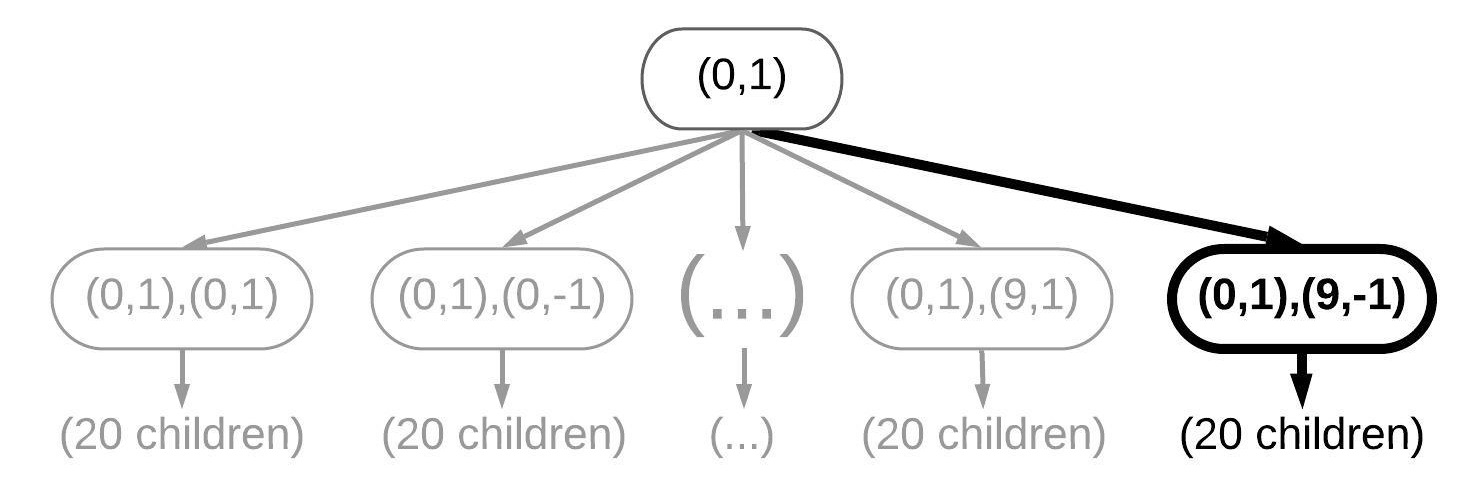}
    \caption{The configuration of $\mathcal{T}$ after the collaboration search up to depth $p=2$. All but the node $(0,1),(9,-1)$ have been eliminated by \hyperlink{U1hyper}{$\bfu^1$-pruning}.}
    \label{depth2_2}
\end{subfigure}\hfill
\caption{Collaboration search at depth $p = 2$}
\label{depth2}
\end{figure*}

To facilitate this search, the key observation is this: from the intersection step, we know that for each $i$, $\bfu_1^i$ and $\bfu_{\max}^i$ must be elements of the equivalent solution $\tV^i$. Thus, at each depth $p$, $\{\Theta^i(\bfu_1^i),\Theta^i(\bfu^i_{\max})\} \subseteq U_p^*$ for all $i = 1,2,\ldots, p$. This means that $U_p^*$ must satisfy the following condition:
\[
\{\bfu_1^1,\bfu_{\max}^1, \Theta^2(\bfu_1^2), \Theta^2(\bfu_{\max}^2), \ldots, \Theta^p(\bfu_1^p), \Theta^p(\bfu_{\max}^p)\} \subseteq U_p^*
\]
Thus for any node $\bfj_p$ in $\mathcal{T}$, if that node or one of its children are $\bfj^*$, it must be that:
\begin{multline}\label{pruning_condition}
\{\bfu_1^1,\bfu_{\max}^1, O_{\bfj_p}^2(\bfu_1^2), O_{\bfj_p}^2(\bfu_{\max}^2), \ldots, O_{\bfj_p}^p(\bfu_1^p), O_{\bfj_p}^p(\bfu_{\max}^p)\} \\ \subseteq \mcC_{\bfj_p}
\end{multline}

We find that (\ref{pruning_condition}) can be implemented as two different algorithmic checks, which are best checked in different ways:
\begin{enumerate}[label=\Roman*.]
    \item \label{U1_pruning}\hypertarget{U1hyper}{\textbf{($\bfu^1$-pruning)}} $$\{\bfu_1^1,\bfu_{\max}^1\}  \subseteq O_{\bfj_p}^i(U^i)$$
\end{enumerate}
Here, $\bfu^1$-pruning depends only on $j_p$, the $p$-th term of the sequence $\bfj_p$. Thus this condition can be checked for the entire set $U^p$ at once due to the following observation: some rearranging gives that $\bfu^1$-pruning is equivalent to $\bfu_{j_p}^p \in (U^p - \bfu_1^1) \cap (U^p - \bfu_{\max}^1)$ if $\omega =1$ and $\bfu_{j_p}^p \in (U^p + \bfu_1^1) \cap (U^p + \bfu_{\max}^1)$ if $\omega = -1$. Thus all possible values of $(j_p,\omega)$ can be found by taking the intersections $\bfu_{j_p}^p \in (U^p - \bfu_1^1) \cap (U^p - \bfu_{\max}^1)$ for $\omega = 1$ and $(U^p + \bfu_1^1) \cap (U^p + \bfu_{\max}^1)$ for $\omega = -1$.

\begin{figure*}[t]
\centering
\begin{subfigure}{.5\textwidth}
    \centering
    \includegraphics[width=1\linewidth]{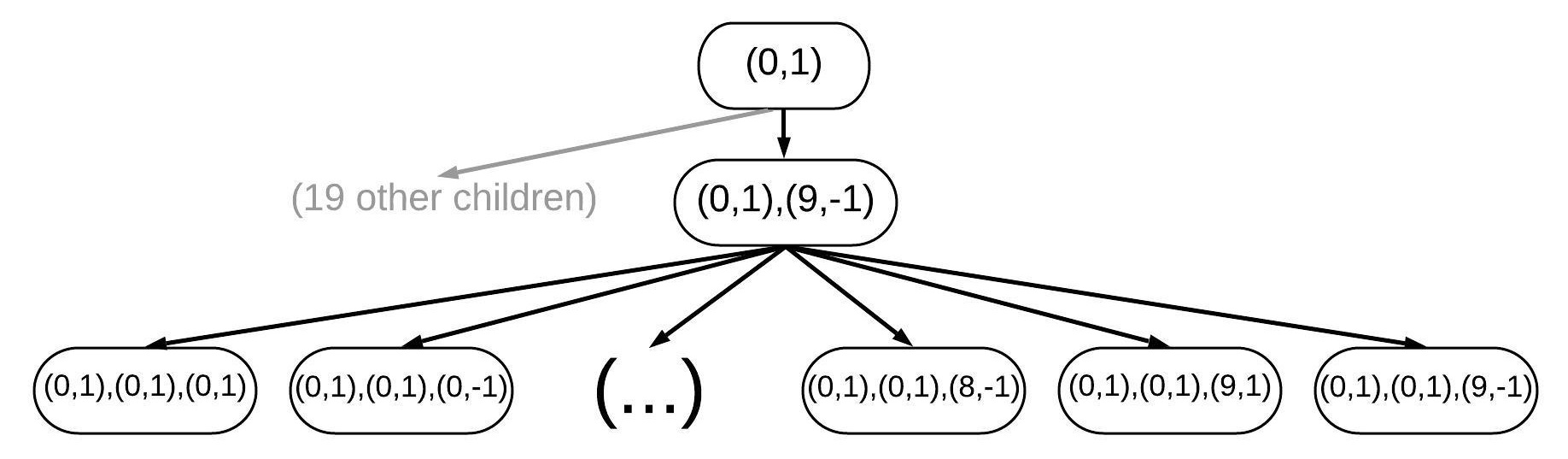}
    \caption{Configuration of $\mathcal{T}$ at depth $p = 3$.}
    \label{depth3_1}
\end{subfigure}\hfill
\begin{subfigure}{.5\textwidth}
    \centering
    \includegraphics[width=1\linewidth]{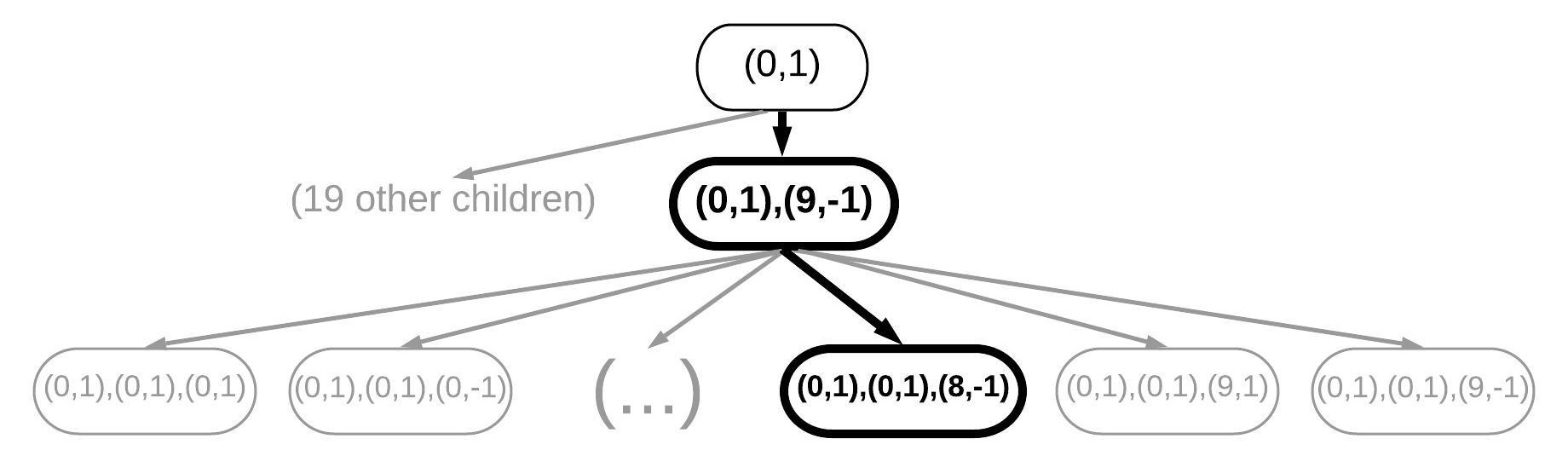}
    \caption{Final configuration of the search tree after collaboration search to depth $p=3$. Only the two nodes in bold ever need to be explored.}
    \label{depth3_2}
\end{subfigure}\hfill
\label{depth3}
\caption{Collaboration search at depth $p=3$.}
\end{figure*}

We demonstrate $\bfu^1$-pruning for $U^2$ in our example. We have $\bfu_1^1 = (1,-2)^T$ and $\bfu_{\max}^1 = (4,1)^T$, where $T$ indicates matrix/vector transpose. We find that:
\[
(U^2 - \bfu_1^1) \cap (U^2 - \bfu_{\max}^1) = \emptyset
\]
\[
(U^2 + \bfu_1^1) \cap (U^2 + \bfu_{\max}^1) = \scriptsize\left\{\begin{pmatrix} 4 \\ 1\end{pmatrix}\right\}
\]\normalsize
It follows that the only valid search index at this depth is $j_2 = (9,-1)$, so we eliminate all other branches of $\mathcal{T}$.

The second condition is the following:
\begin{enumerate}[resume, label=\Roman*.]
    \item\label{Ui_pruning} \hypertarget{Uihyper}{\textbf{(Multiple $\bfu^i$-pruning)}}
\end{enumerate}
\begin{multline*}\hspace{-5pt}
    \{O_{\bfj_p}^2(\bfu_1^2), O_{\bfj_p}^2(\bfu_{\max}^2), \ldots, O_{\bfj_p}^p(\bfu_1^p), O_{\bfj_p}^p(\bfu_{\max}^p)\} \subseteq O_{\bfj_p}^i(U^i)
\end{multline*}

Unlike $\bfu^1$-pruning, multiple $\bfu^i$-pruning depends on the entire parent sequence $\bfj_{p-1}$ and so needs to be checked separately for each $\bfj_{p-1}$ remaining in the tree. This condition is empty for depth $p=2$, so it trivially holds in our example.

Next, if a node $\bfj$ survives both $\bfu^1$-pruning and multiple $\bfu^i$-pruning, we ``explore'' the node by computing the collaboration $\mcC_\bfj$. If $\mcC_\bfj$ contains our target solution $\tV^1$, then its difference set $\diff{\mcC_\bfj}$ must contain $W$. This informs the following checks: first, it must be that $|\mcC_\bfj| \geq k_{\min}$, as $\mcC_\bfj$ must contain enough elements to generate a difference set the same size as $W$. If this holds, we could check explicitly if $W \subseteq \diff{\mcC_\bfj}$; however, computing $\diff{\mcC_\bfj}$ and checking this condition takes $\bigO{|\mcC_\bfj|^2\log|\mcC_\bfj|}$ operations. Since $|\mcC_\bfj|$ can be as large as $\bigO{k^2}$, we only perform this check if $|\mcC_\bfj| < c k_{\min}$ where $c > 1$ is a predetermined constant. This guarantees we only check $W \subseteq \mcC_\bfj$ if doing so will take at most $\bigO{k^2\log{k}}$ time. (Details on the selection of this constant are included in subsection \ref{param_c}; $c = 2$ is our default choice.) We refer to the subtree of $\mathcal{T}$ consisting of all nodes explored in this process as the \textit{explored tree.}

We have now covered all the steps in the collaboration search, and so we can apply them to conclude our example. Having confirmed that $j_2 = (9,-1)$ is the only node to survive the pruning conditions, we compute the collaboration with respect to $\bfj_2 = (0,1), (9,-1)$:
\begin{multline*}
\mcC_{\bfj_2} = U^1 \cap - \left(U^2 - \scriptsize\begin{pmatrix} 4 \\ 1\end{pmatrix}\normalsize\right) =\\= \scriptsize\left\{\begin{pmatrix} 0 \\ 0\end{pmatrix}, \begin{pmatrix} 1 \\ -2\end{pmatrix}, \begin{pmatrix} 1 \\ 0 \end{pmatrix}, \begin{pmatrix} 2 \\ -1\end{pmatrix}, \begin{pmatrix} 3 \\-1\end{pmatrix}, \begin{pmatrix} 2 \\2\end{pmatrix},\begin{pmatrix} 3 \\ 0\end{pmatrix},\begin{pmatrix} 4 \\ 1\end{pmatrix}\right\}\normalsize
\end{multline*}
This set has $k < 8 < 2k$ elements, so we compute $\diff{\mcC_{\bfj_2}}$ and find that $W \subset \diff{\mcC_{\bfj_2}}$. Thus this node remains, but is not itself a solution, so we must continue to depth $p = 3$. The configuration of $\mathcal{T}$ after this step can be seen in figure \ref{depth2_2}.

At depth $p=3$, $\mathcal{T}$ has initial setup shown in figure \ref{depth3_1}. We begin by checking condition (\ref{U1_pruning}):
\[
(U^3 - \bfu_1^1) \cap (U^3 - \bfu_{\max}^1) = \emptyset
\]
\[
(U^3 + \bfu_1^1) \cap (U^3 + \bfu_{\max}^1) = \scriptsize\left\{\begin{pmatrix} 2 \\ 2\end{pmatrix}\right\}\normalsize
\]
This means that the only valid index at this depth is $(7,-1)$. We then confirm that condition (\ref{Ui_pruning}) holds, i.e. that $(3,-1)^T$ is in $-\left(U^3 - (2,2)^T\right)$. Having passed conditions (\ref{U1_pruning}) and (\ref{Ui_pruning}), we compute the collaboration $\mcC_{\bfj_3}$ for $\bfj_3 = (0,1),(9,-1),(7,-1)$: 
\begin{multline*}
\mcC_{\bfj_3} = \mcC_{\bfj_2} \cap -\left(U^3 - \scriptsize\begin{pmatrix} 2 \\ 2\end{pmatrix}\normalsize\right)=\\= \scriptsize\left\{\begin{pmatrix} 0 \\ 0\end{pmatrix}, \begin{pmatrix} 1 \\ -2\end{pmatrix}, \begin{pmatrix} 1 \\ 0 \end{pmatrix}, \begin{pmatrix} 3 \\-1\end{pmatrix}, \begin{pmatrix} 2 \\2\end{pmatrix},\begin{pmatrix} 3 \\ 0\end{pmatrix},\begin{pmatrix} 4 \\ 1\end{pmatrix}\right\}\normalsize
\end{multline*}

$|\mcC_{\bfj_3}|$ is exactly 7, so we compute the difference set $\diff{\mcC_{\bfj_3}}$ and confirm that $\diff{\mcC_{\bfj_3}} = W$, and thus $\mcC_{\bfj_3}$ is a solution. The final configuration of $\mathcal{T}$ is shown in figure \ref{depth3_2}, while figure \ref{exampleFig} shows the set $V$ alongside the recovered set $\mcC_{\bfj_3}$.

\begin{figure}
\begin{subfigure}{.25\textwidth}
    \centering
    \includegraphics[width=.98\linewidth]{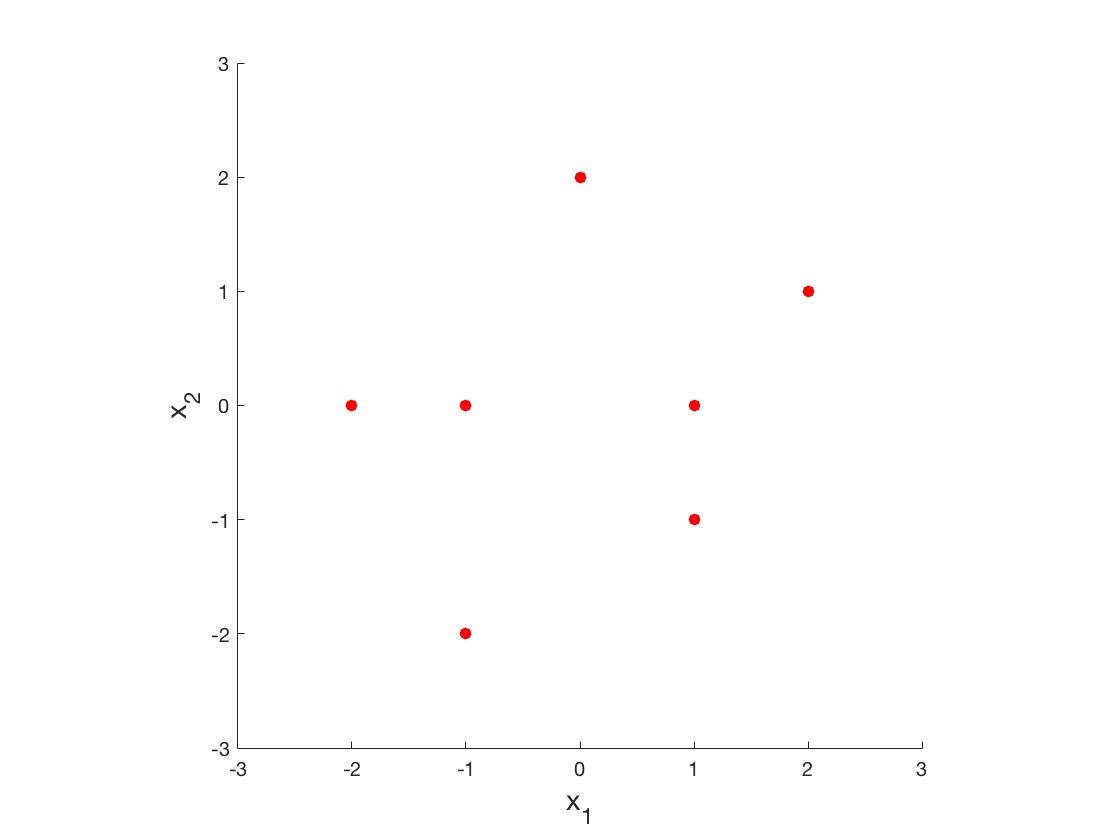}
    \caption{Original scatterers $V$}
    \label{exampleV}
\end{subfigure}%
\begin{subfigure}{.25\textwidth}
    \centering
    \includegraphics[width=.98\linewidth]{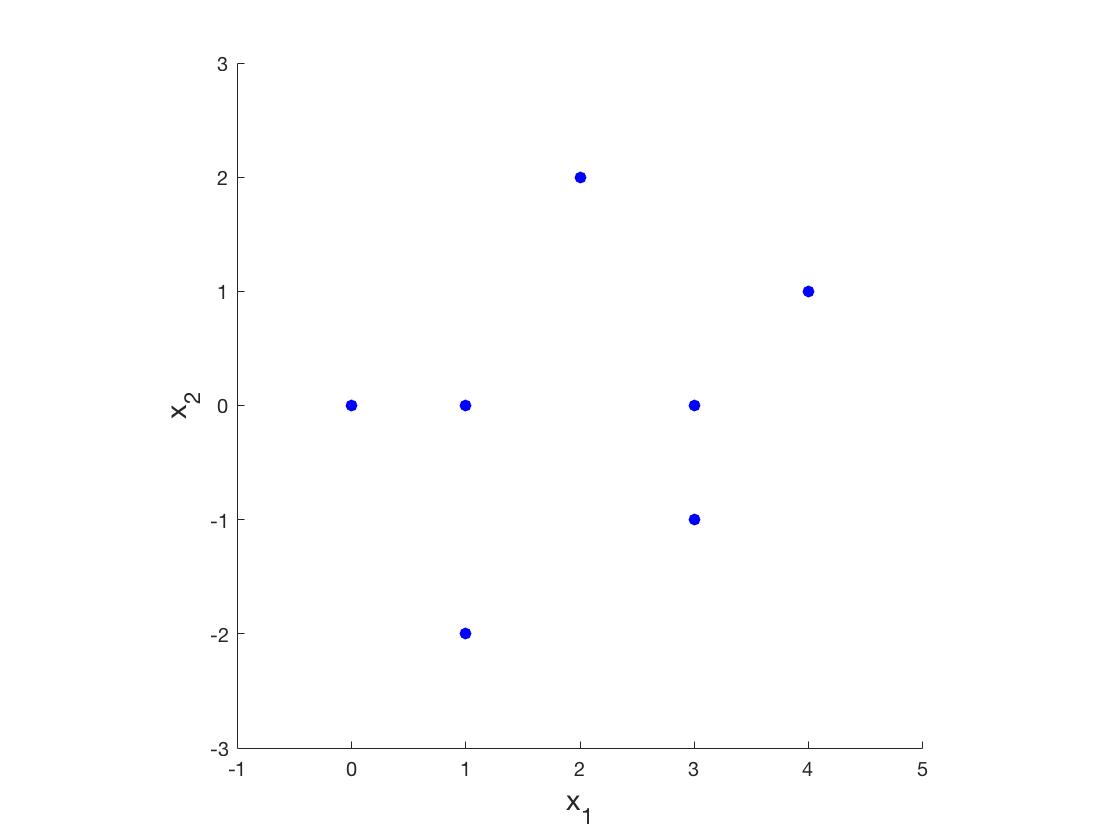}
    \caption{Retrieved set $\mcC_{\bfj_3}$}
    \label{exampleC}
\end{subfigure}
\caption{Original set of scatterers $V$ alongside recovered scatterer locations $\mcC_{\bfj_3}$. The two differ only by a constant shift.}
\label{exampleFig}
\end{figure}

\subsection{Collaboration Search: Summary}
The collaboration search can be summarized as follows: for each $p = 2,\ldots, \uptau$, only retain the children of search sequences $\bfj$ whose corresponding orientations $O_{\bfj_p}^i(U^i)$ and collaborations $\mcC_{\bfj_p}$ satisfy the conditions:
\begin{enumerate}[label=\Roman*.]
    \item $O_{\bfj_p}^i(U^i) \supseteq \{\bfu_1^1,\bfu_{\max}^1\}$
    \item $O_{\bfj_p}^i(U^i) \supseteq \{O_{\bfj_p}^2(\bfu_1^2), O_{\bfj_p}^2(\bfu_{\max}^2), \smalldots O_{\bfj_p}^p(\bfu_1^p), O_{\bfj_p}^p(\bfu_{\max}^p)\}$
    \item $\diff{\mcC_{\bfj_p}} \supseteq W$
\end{enumerate}

(The third condition is checked only if $|\mcC_\bfj| < ck_{\min}$.) If $W = \diff{\mcC_{\bfj_p}}$ exactly, return $\mcC_{\bfj_p}$ as a solution; otherwise, proceed to depth $p+1$. Pseudocode outlining our implementation of the collaboration search is provided in figure \ref{collabStep}.

When $p = \uptau$, one of the remaining nodes will always be $U$; it follows that the only case in which the collaboration search fails to eventually find a solution is the case that $U$ is not itself a solution because it contains a false positive vector.  We prove in the results section that for $\theta < 1/2$, the probability that this occurs is small for sufficiently high resolution $n$. In the event that exact recovery fails, none of the leaves of $\mathcal{T}$ have difference set exactly $W$, we return the sparsest collaboration among leaves which are not eliminated during the collaboration search. Formally, this is the collaboration of the search sequence $\bfj^{best}$ according to the rule:
$$\bfj^{best} = \argmin_{\bfj \in \mathbf{leaves(\mathcal{T})}} \{|\mcC_\bfj| : W \subseteq \diff{\mcC_\bfj}\}$$

\subsection{Choice of the parameter $c$}\label{param_c}
This constant is included to ensure that we only compute $\diff{\mcC_\bfj}$ when $|\diff{\mcC_\bfj}| = \bigO{k}$, and thus it will take at most $\bigO{k^2\log(k)}$ time to compare $\diff{\mcC_\bfj}$ to $W$. Higher $c$ corresponds to slower node exploration time in exchange for more nodes pruned from the tree each step. Our experiments indicate that a smaller $c$ is typically faster, but if $c$ is too small then it is possible that $ck_{\min} < k$, in which case the algorithm will prune even the correct nodes in $\mathcal{T}$ and fail. We found that $c = 2$ strikes an appropriate balance.

\subsection{Time complexity}

We now consider the time complexity of the collaboration search. Since $U^i \subset W$, $|U^i| < k^2$ for each $i$. Thus condition (\ref{pruning_condition}) can be checked by a sort and binary search in $\bigO{k^2\log(k)}$ time. Computing $\mcC_\bfj$ involves a single intersection for each node passing (\ref{pruning_condition}), which takes at most $\bigO{k^2 \log(k)}$ time for each node, and the restriction on computing difference sets only if $|\mcC_\bfj| < ck_{\min}$ guarantees that this check takes at most $\bigO{k^2 \log(k)}$ time to compute the difference set and compare it with $W$. Thus as long as the number of indices explored for each depth $p$ is not too large, for each $p$ the search and exploration process can be completed in $\bigO{k^2 \log(k)}$ time. 

In particular, if the explored tree never branches---that is, only a single node survives (\ref{pruning_condition}) at each depth---then the entire collaboration search can be performed in $\bigO{\uptau k^2 \log(k)}$ worst-case time, the same as the time complexity of the projection step. If $\theta < 1/2$, branching occurs with low probability. Indeed, by similar combinatorial logic to that used to prove our main theorem, when $\theta < 1/2$ the probability that the explored tree branches even once tends to zero as $n \to \infty$. Thus while the worst-case of this search is, as the backtracking algorithm of Lemke et al. \cite{LeSkSm:2003}, exponential, below the theoretical recovery threshold it is rare for even a single branch to occur in the explored tree, let alone many. The resulting $\bigO{\uptau k^2 \log(k)}$ runtime is confirmed by our numerical results in section \ref{5_numerical}.

\begin{algorithm*}[h]
\caption{Collaboration Search Step} \label{collabStep}
\begin{algorithmic}[1]
\State $J(1) = \{(0,1)\}$
\For{$p = 2:\uptau$}
    \State $J(p) = \emptyset$
    \State Compute $\mcI(p)$, the set of all indices $(j,\omega)$ at depth $p$ which pass condition (\ref{U1_pruning}):
    \State $\mcI(p) = \big\{(j, \omega) \mid j \in \{0,1,2,\ldots, |U^p|-1\},$ $\omega \in \{1,-1\},$ $\{\bfu_1^1,\bfu_{\max}^1\} \subseteq \omega(U^p - \bfu^p_j)\big\}$
    \For{$(j,\omega) \in \mcI(p)$}
        \For{$\bfj_{p-1} \in J(p-1)$}
            \State $\bfj_p = \bfj_{p-1},(j,\omega)$
            \If{$\{O_{\bfj_{p-1}}^2(\bfu_1^2), O_{\bfj_{p-1}}^2(\bfu_{\max}^2), \ldots, O_{\bfj_{p-1}}^p(\bfu_1^p), O_{\bfj_{p-1}}^p(\bfu_{\max}^p)\} \subseteq \omega(U^p - \bfu^p_j)$}
                \State $\mcC_\bfj = \mcC_{\bfj_{p-1}} \cap \omega(U^p - \bfu^p_j)$
                \If{$|\mcC_{\bfj_d}| \geq k_{\min}$}
                    \If{$|\mcC_{\bfj_d}| > ck_{\min}$}
                        \State Add $\bfj_d$ to $J(p)$
                        \State (Don't compute $\diff{\mcC_{\bfj_p}}$)
                    \Else
                         \If{$W = \diff{\mcC_{\bfj_p}}$}{
                            \Return $\{\mcC_{\bfj_p}, \text{TRUE}\}$
                        } \ElsIf {$W \subseteq \diff{\mcC_{\bfj_p}}$}
                            \State Add $\bfj_p$ to $J(p)$
                        \EndIf
                    \EndIf
                \EndIf 
            \EndIf
        \EndFor
    \EndFor
\EndFor\\
If loop completes, no solution was found: 
\State $\bfj^{best} = \argmin_{\bfj \in \mathbf{leaves(\mathcal{T})}} \{|\mcC_\bfj| : W \subseteq \diff{\mcC_\bfj}\}$\\
\Return $\{\mcC_{\bfj^{best}}, \text{FALSE}\}$
\end{algorithmic}
\end{algorithm*}

\subsection{Implementation Details}
In algorithm 1, we describe MISTR with all $\uptau$ projection and intersection steps finished before attempting a collaboration search. In practice it is faster to run these steps in a different order: for each $i$, perform a projection step, an intersection step, and a collaboration search for depth $p = i$, proceeding to $i+1$ if no solution is found. This avoids performing a large number of projection steps when only a small number are required for the collaboration search to find a solution. We use this order in our numerical simulations in section \ref{5_numerical}.


\section{Theoretical Results}\label{4_results}
This section provides an overview of the proof of our main result. Full proofs are available in the appendix.

\subsection{Probabilistic Setting}

We now introduce the probabilistic setting for which we will prove our theoretical guarantees. We consider an underlying continuous process for distributing points, and show how it can be reduced to a discrete setting in a natural way. 

If scatterers are modeled as points with zero volume, then a single intersection step recovers $V_s$ with probability one; interesting phenomena occur only when scatterers have nonzero size. We choose to model these as pixels of nonzero finite size for theoretical tractability, but any reasonable model should yield similar results.

Let $s = n^{d\theta}$ where $0<\theta<1/2$. Let $\Phi_s$ denote the probability measure associated with a Gaussian random variable in $\mathbb{R}^d$ with mean $0$ and covariance matrix $\frac{1}{2\ln(s)}I$; this is the distribution of a standard Gaussian normalized by $1/\sqrt{2\ln{(s)}}$. We consider a Poisson point process $\Pi_s$ with intensity $s$ and underlying measure $\Phi_s$. Formally, for any set $A \subseteq \mathbb{R}^d$, the number of points $|A \cap \Pi_s|$ is distributed as a Poisson random variable with mean $s\Phi_s(A)$, and these distribution are independent for any collection of disjoint sets.

We include the normalization factor of $1/\sqrt{2\ln(s)}$ to guarantee that $\max_{\bfp \in \Pi_s}\|\bfp\|_2$ tends to 1 with high probability as $n \to \infty$. From an imaging perspective, this normalization ensures the points in $\Pi_s$ lie inside the same image window for all sufficiently large $n$ (with high probability).

Recall that $\mathbb{Z}_n^d = \left(\frac{1}{n}\mathbb{Z}\right)^d$ and let $\bfc \in \mathbb{Z}_n^d$. Now, denote $P_n(\bfc)$ the open $\ell^\infty$ ball centered at $\bfc$ with radius $1/(2n)$; we will call these balls ``pixels.'' Our goal is to recover a discrete process $V_s$ on $\mathbb{Z}_n^d$ from its pairwise difference set, where $V_s$ is distributed as follows:
\[
V_s = \{\bfc \in \mathbb{Z}_n^d : |P_n(\bfc) \cap \Pi_s| \geq 1\}
\]
$V_s$ is the natural discretization of $\Pi_s$ into cubic pixels.

With probability 1, all points in $\Pi_s$ will lie in an $\ell^\infty$ ball centered at a point in $\mathbb{Z}_n^d$, so no points from $\Pi_s$ are missed by this discretization process. It is clear that the pixels are disjoint, so the probabilities that distinct points $\bfc_1, \bfc_2\in \mathbb{Z}_n^d$ are in $V_s$ are independent. By this independence property and the definition of Poisson point process, we have that each $\bfc \in \mathbb{Z}_n^d$ is in $V_s$ independently with probability: 
\begin{multline*}
    \prob(\bfc \in V_s) = 1 - \exp(-s\Phi_s(P_n(\bfc))) = \\ = s\Phi_s(P_n(\bfc)) + \bigO{(s\Phi_s(P_n(\bfc)))^2}
\end{multline*}
In this setup, $n$ should be understood as representing resolution, and $n^d$ will be proportional to the number of pixels in a fixed image window. With our first lemma, we remark that for sufficiently high $n$, no two points in $\Pi_s$ will fall in the same pixel with high probability. For all lemmas and theorems, we work in $\mathbb{R}^d$ for $d \geq 2$ and with $n > 0$.

\begin{restatable}{lemma}{twoPointLemma}\label{twoPointLemma}
Let $\theta \in (0,1/2)$, and let $\Pi_s$ be a Poisson point process on $\mathbb{R}^d$ with intensity $s$ and underlying measure $\Phi_s \sim \mcN\left(\bfo, \frac{I}{2\ln(s)}\right)$. Then for large enough $n$, the probability that there exists $\bfc^* \in \mathbb{Z}_n^d$ such that $|P_n(\bfc^*) \cap \Pi_s| \geq 2$ is bounded by $\frac{2}{s^{1/C}} + \bigO{\frac{s^2\ln^{d}(s)}{n^{d}}}$ for $C$ an absolute positive constant.
\end{restatable}
Thus if $s = n^{d\theta}$ for $\theta < 1/2$, the discrete process $V_s$ contains the same number of points as the underlying process $\Pi_s$ with high probability as $n \to \infty$.

We choose to study this specific model for its theoretical tractability---in particular its independence properties and rotational invariance---but our results generalize to any natural choice of distribution under the theoretical threshold. We include some numerical results for other distributions in our supplementary media, which show good recovery performance below the $\theta < 1/2$ threshold.

\subsection{Main Results}

We now state our main result, which holds that the probability that MISTR fails for the discrete process $V_s$ becomes arbitrarily small for large $n$ as long as $s$ grows algebraically slower than $n^{d/2}$:

\begin{restatable}{theorem}{gaussTheorem}
\label{gaussTheorem}
Let $s = n^{d\theta}$ be a sparsity parameter, with $\theta \in (0,1/2)$. Let $V_s$ be a random subset of $\mathbb{Z}_n^d$ such that each $\bfc \in \mathbb{Z}_n^d$ is in $V_s$ independently with probability $1-\exp(s\Phi_s(P_n(\bfc)))$, where $\Phi_s$ is the measure associated with a normal $\mcN\left(0,\frac{I}{2\ln(s)}\right)$ random variable in $\mathbb{R}^d$. Then MISTR with $\uptau > \frac{2}{(1/2 - \theta)^2}$ random projections recovers a set equivalent to $V_s$ with probability tending to 1 as $n \to \infty$.
\end{restatable}

This theorem is asymptotic in the resolution $n$. Our numerical results in section \ref{5_numerical} show that MISTR achieves excellent results with relatively few ($\lesssim 30$) random projections.

We begin with an analysis of what conditions must hold for MISTR to fail. In what follows we make the simplifying assumption that $\tV^i = V_s - \bfv_0^i$ for each $i = 1,\ldots,\uptau$. This ensures that $\tbfv_{1}^i = \bfv_1^i - \bfv_0^i$. In the case where $\tV^i = \bfv_{k-1}^i - V_s$ for some subset of $i$, one can apply the same arguments working from $\bfv_{k-1}^i$ instead of $\bfv_0^i$. Since $U^i = W^i \cap \left(W^i + \bfv_1^i - \bfv_{0}^i\right)$, this failure condition is equivalent to:
\begin{equation}
    \bfl \notin V_s, \bfl - \bfv_0^1 \in \bigcap_{i=1}^\uptau \left\{\left[W^i \cap \left(W^i + \bfv_{1}^i - \bfv_0^i\right)\right]+\bfv_0^i - \bfv_0^1\right\}
\end{equation}
After some rearranging, this is equivalent to:
\begin{equation}\label{failCond}
    \bfl \not\in V_s, \{\bfl - \bfv_0^i, \bfl - \bfv_1^i\} \subseteq W^i, i = 1,2,\ldots, \uptau
\end{equation}
As this event is difficult to express concisely, we denote by $\mfF_\bfl$ the event that $\{\bfl - \bfv_0^i, \bfl - \bfv_1^i\} \subseteq W$ for each $i = 1,2,\ldots, \uptau$. Since each element in $W^i$ is either the positive or negative multiple of every element in $W$, this is equivalent to $\{\bfl - \bfv_0^i, \bfl - \bfv_1^i\} \subseteq W^i$. Thus the event that MISTR fails (i.e., returns a false positive) can be expressed mathematically as:
\begin{equation}
    \{\text{MISTR fails}\} = \bigcup_{\bfl \in \mathbb{Z}_n^d}\{\bfl \notin V_s\} \cap \mfF_\bfl
\end{equation}

Denote by $V_\uptau$ the multiset (set with possibly repeated elements) defined by:
\begin{equation}\label{vuptau}
V_\uptau = \{\bfv_0^1,\bfv_1^1,\bfv_0^2,\bfv_1^2,\ldots, \bfv_0^\uptau, \bfv_1^\uptau\}
\end{equation}
Denote $\dist{V_\uptau}$ be the set of distinct elements in $V_\uptau$. From (\ref{failCond}) the number of different conditions that $\bfl$ must satisfy in order to cause the algorithm to fail is exactly $|\dist{V_\uptau}|$. We quantify this statement with the following lemma:

\begin{restatable}{lemma}{allPointsLemmaGauss}
\label{all_points_lemma_Gauss}
Let $\uptau \in \mathbb{N}$, $\theta \in (0,1/2)$, and $s = n^{d\theta}$. Let $\{\bfz^j\}_{j=1}^\uptau$ be independent, uniformly distributed random vectors on $\mathbb{S}^{d-1}$. Let $V_\uptau$ be as defined in (\ref{vuptau}), and suppose that $|\dist{V_\uptau}| = t \in \mathbb{N}$. Then for large enough $n$ and absolute positive constant $C$, the probability that there exists $\bfl \in \mathbb{Z}_n^d$, such that $\bfl \notin V_s$ and $\mfF_{\bfl}$ occurs, is bounded by:
\begin{equation}
\frac{2}{n^{d\theta/C}}+\bigO{\frac{\ln^{d\sqrt{t}/4}(n)}{n^{d\left(\frac{(1/2-\theta) \sqrt{t}}{2} - 1\right)}}}
\end{equation}
\end{restatable}

As the first term tends to zero for any $\theta > 0$, convergence of the error probability to 0 is determined by the second term. If $\theta < 1/2$ and $t > 2/(1/2-\theta)^2$, this term tends to 0 as $n \to \infty$. It follows that if $|\dist{V_\uptau}|$ becomes large with sufficiently high probability, then error probability will become arbitrarily small as $n \to \infty$ for any $\theta < 1/2$. To do this, we show that the set $\{\bfv_0^i\}_{i=1}^\uptau$ will equal $\uptau$ with probability arbitrarily close to 1 as $n \to \infty$. Denote by $\conv{V_s}$ the convex hull of $V_s$. We introduce the following definition:

\begin{definition}[Minimal Point]
A point $\bfv\in V_s$ is a \textbf{minimal point} of $V_s$ if there exists a nonempty open set $A$ of $\mathbb{S}^{d-1}$ such that for every $X \in A$, $\inner{X}{\bfv} = \min_{\bfv'\in V_s}\inner{X}{\bfv'}$.
\end{definition}
 Another way to interpret this definition is that $p$ is a minimal point iff with nonzero probability, it has the minimal projection in $V_s$ with a vector chosen uniformly at random from the half-sphere $\mathbb{S}^{d-1}$. It is clear that the minimal points of $V_s$ comprise a subset of $\conv{V_s}$. We will denote the set of vertices of $\conv{V_s}$ as $\Vertex{V_s}$.
 
 Thus the distribution of $\dist{V_\uptau}$ is closely related to the distribution of the vertices $\Vertex{V_s}$. We now introduce a corollary of a theorem of Davydov \cite{Da:2011} regarding the convex hulls of i.i.d. copies of Gaussian processes. Let $C^k$ be the convex hull of the process consisting of $k$ i.i.d. Gaussian random vectors with mean 0 and covariance $I$. Then:

\begin{corollary}[Davydov]\label{corollaryDavydov}
As $k \to \infty$, with probability 1,
\[
\frac{1}{\sqrt{2\ln(k)}}C^k \to B(0,1)
\]
in the Hausdorff metric.
\end{corollary}

It follows from the corollary that $\Vertex{V_s}$ converges to the unit sphere $\mathbb{S}^{d-1}$. Thus for large enough $n$, $\Vertex{V_s}$ is nearly uniformly distributed on $\mathbb{S}^{d-1}$, from which we can conclude:
\begin{restatable}{lemma}{distProbLemma}
\label{distProbLemma}
Let $s = n^{d\theta}$ for $\theta \in (0,1/2)$. Let $V_\uptau$ be defined as in equation (\ref{vuptau}). Then $\prob\left(|\dist{V_\uptau}| \geq \uptau \right)\to 1$ as $n \to \infty$.
\end{restatable}

Theorem \ref{gaussTheorem} then follows from this and Lemma \ref{all_points_lemma_Gauss}.

\subsection{Variations on Theorem \ref{gaussTheorem}}\label{variations}

We remark on possible variations of the results proved above. Our proofs of this result do not fundamentally rely on our choice of a Gaussian distribution for our underlying Poisson point process. Indeed, our results hold essentially unchanged were $\Pi_s$ replaced by any point process $\widetilde{\Pi}_s$ such that $\conv{\widetilde{\Pi}_s}$ approaches a compact convex set with everywhere positive curvature as $s \to \infty$. This holds in expectation for $s$ points uniformly distributed in a convex set $A$ with smooth boundary and everywhere positive curvature, in which case the expected Hausdorff distance between $A$ and the convex hull of $s$ points distributed uniformly in $A$ is approximately $\left(\ln(s)/s\right)^{2/(d+1)}$ \cite{Ba:1989}. This class includes the natural model of i.i.d. points in a unit sphere, which we consider the appropriate multidimensional analogue to a uniform distribution on a line.

By contrast, technical complications arise when the limit shape of $\conv{V_s}$ has zero curvature or has corners, including the seemingly natural model of a uniform Poisson point process on a $d$-dimensional cube. In this case, the vertices of $\conv{V_s}$ do not collect roughly uniformly along the boundary of some set. Instead, vertices collect close to the corners of the cube as $n$ grows to infinity. This makes it difficult to prove analogous results to Lemma \ref{distProbLemma}. Empirically, though, MISTR still enjoys excellent recovery properties for this distribution below the $\theta < 1/2$ threshold.

\subsection{Robustness to Noise}
Consider the case where the autocorrelation is corrupted by Gaussian noise $\bfe$. Then we observe $\bfy = \bfa + \bfe$, which can be understood as a vector in $\mathbb{R}^M$ for some $M$. In this setting, the most logical means of estimating $\supp{\bfa}$ is by normalizing and thresholding the vector $\bfy$: estimate $\supp{\bfa}$ by the set $S_\tau = \{i : y_i/\|\bfy\|_2 > \tau\}$ for a parameter $\tau$. As long as $S_\tau = \supp{\bfa}$, then MISTR can be applied to recover $\supp{\bfx}$ as if in the noiseless case. 

If $s=n^{d\theta}$ for $\theta < 1/4$, we can derive sufficient conditions for this event to occur with high probability as follows. When $\theta < 1/4$, elements in $\supp{\bfx}$ have unique pairwise differences with probability tending to 1 as $n \to \infty$; this implies that if there exist constants $c_1, c_2$ such that $c_1 < |x_i| < c_2$ for all $i \in \supp{\bfx}$, then likewise $c_1^2 < |\bfa_j| < c_2^2$ for all $j\in \supp{\bfa}$. For a preselected error probability $\varepsilon$, we can set $\tau = \frac{\sqrt{2\ln(M)}+\sqrt{2\ln(1/\varepsilon)}}{\sqrt{M}}$ which by the Borell-TIS inequality (see, e.g., \cite{Ad:2007}) guarantees that $S_\tau \subseteq \supp{\bfa}$ with probability at least $1-\varepsilon$. Setting $\kappa = |\supp{\bfa}|$, one can then prove that if the noise level $\sigma = \|\bfe\|_2/\|\bfa\|_2$ satisfies
\[
\sigma < \frac{c_1\sqrt{M}}{2c_2\sqrt{\kappa}\left(\sqrt{2\ln(M)}+\sqrt{2\ln(2/\varepsilon)}\right)}
\]
then for large enough $n$, $S_\tau = \supp{\bfa}$ with probability at least $1-\varepsilon$. 

When $\theta \geq 1/4$, however, more than one pair of elements in $\supp{\bfx}$ may have the same difference with non-negligible probability, called a ``collision.'' When this happens, cancellations can occur in the resulting autocorrelation. Deriving the maximum noise level for recovery in this setting requires a more detailed analysis which is beyond the scope of our initial work. That said, our numerical simulations detailed in figure \ref{PRsupportFIG} demonstrate recovery well above this cutoff for several levels of noise.

\section{Numerical Simulations}\label{5_numerical}

\begin{figure*}
\begin{subfigure}{.5\textwidth}
  \centering
  \includegraphics[width=.9\linewidth]{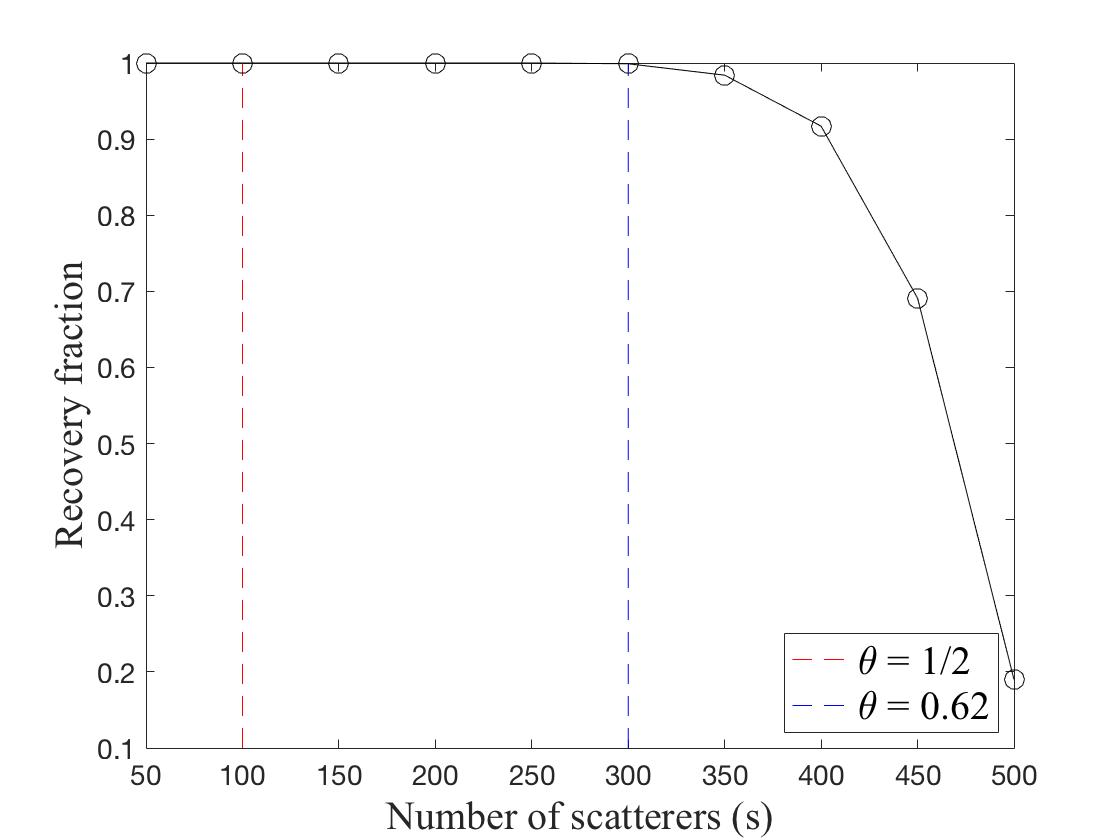}
  \caption{Recovery probability v. number of scatterers}
  \label{probSparsityFig}
\end{subfigure}%
\begin{subfigure}{.5\textwidth}
  \centering
  \includegraphics[width=.9\linewidth]{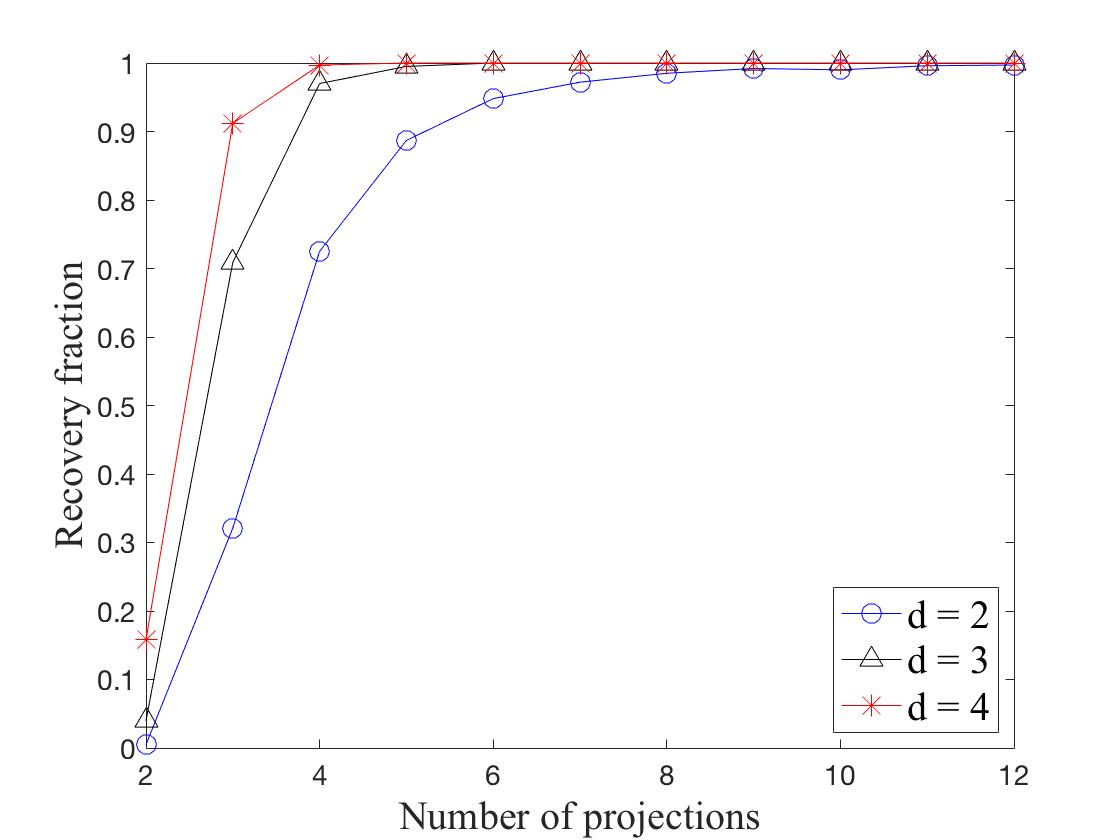}
  \caption{Recovery probability v. number of projections}
  \label{probProjFig}
\end{subfigure}
\caption{\textbf{Recovery probability.} (a) shows how recovery probability varies with $s$ when $\uptau = 30$, $d = 3$ and $n^d=10000$. (b) shows recovery probability against $\uptau$ for $s = 200$, $n^d = 40000$ in dimensions 2, 3, and 4.}
\label{probFig}
\end{figure*}

\begin{figure*}
\begin{subfigure}{.5\textwidth}
  \centering
  \includegraphics[width=.9\linewidth]{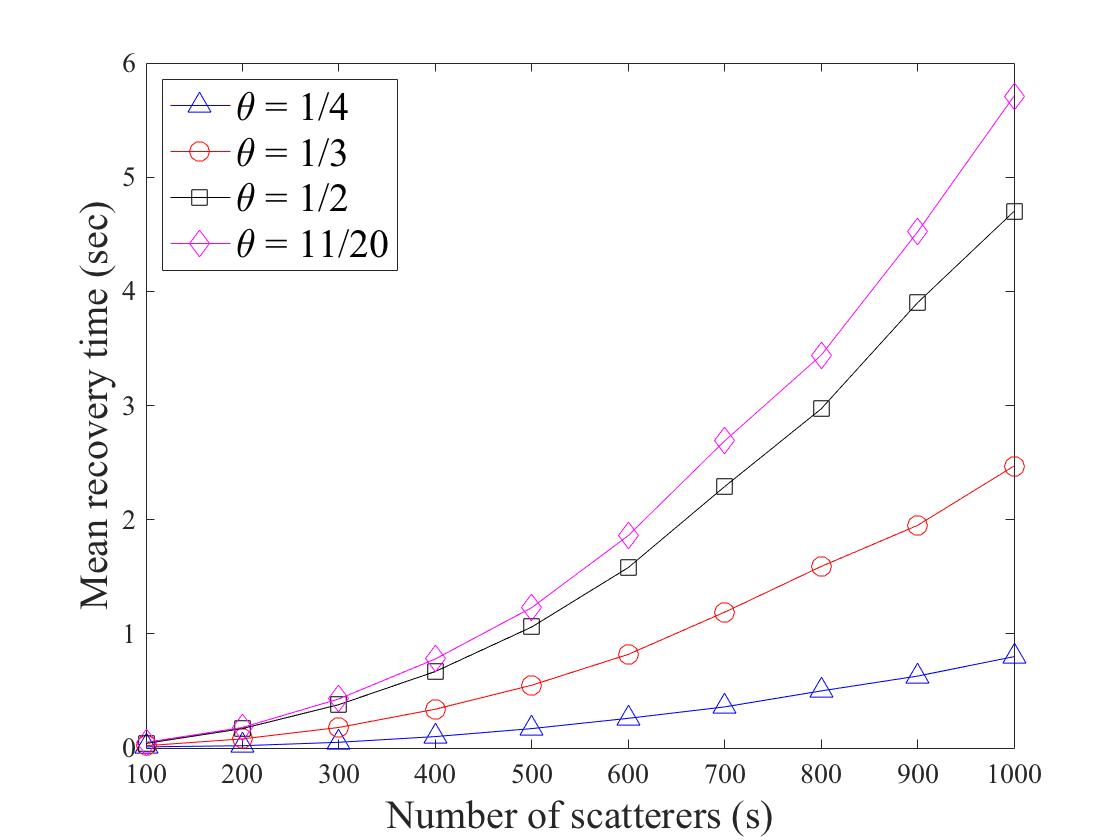}
  \caption{Recovery time v. number of scatterers}
    \label{timeRelFigD3}
\end{subfigure}
\begin{subfigure}{.5\textwidth}
  \centering
  \includegraphics[width=.9\linewidth]{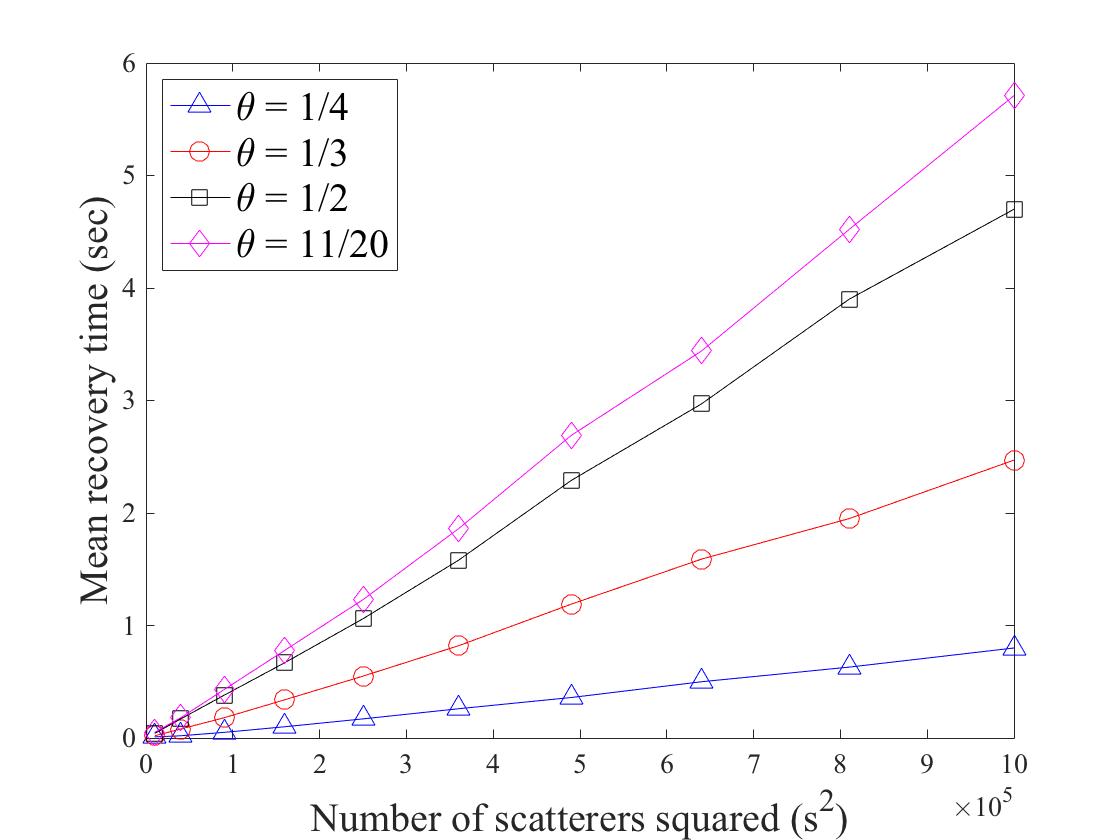}
  \caption{Recovery time v. $s^2$}
    \label{timeSparsitySqrdFig}
\end{subfigure}
\begin{subfigure}{.5\textwidth}
  \centering
  \includegraphics[width=.9\linewidth]{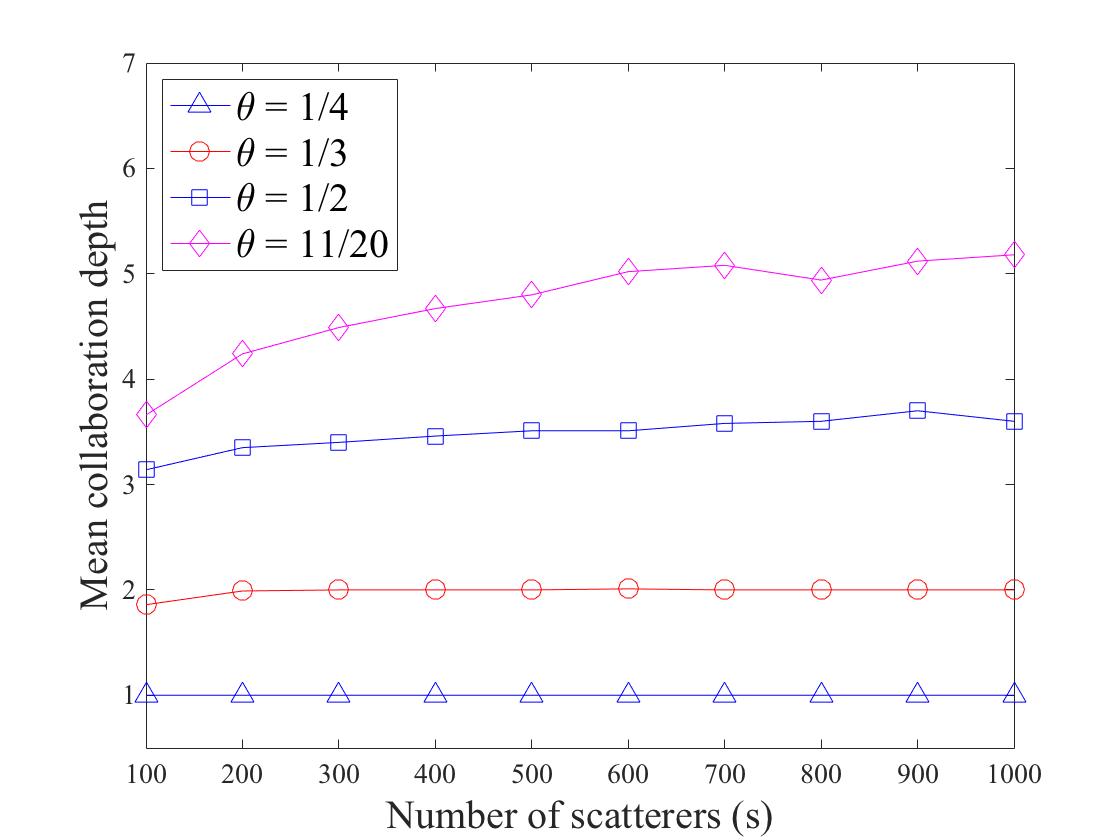}
  \caption{Collaboration depth v. number of scatterers}
    \label{collabFig}
\end{subfigure}
\begin{subfigure}{.5\textwidth}
  \centering
  \includegraphics[width=.9\linewidth]{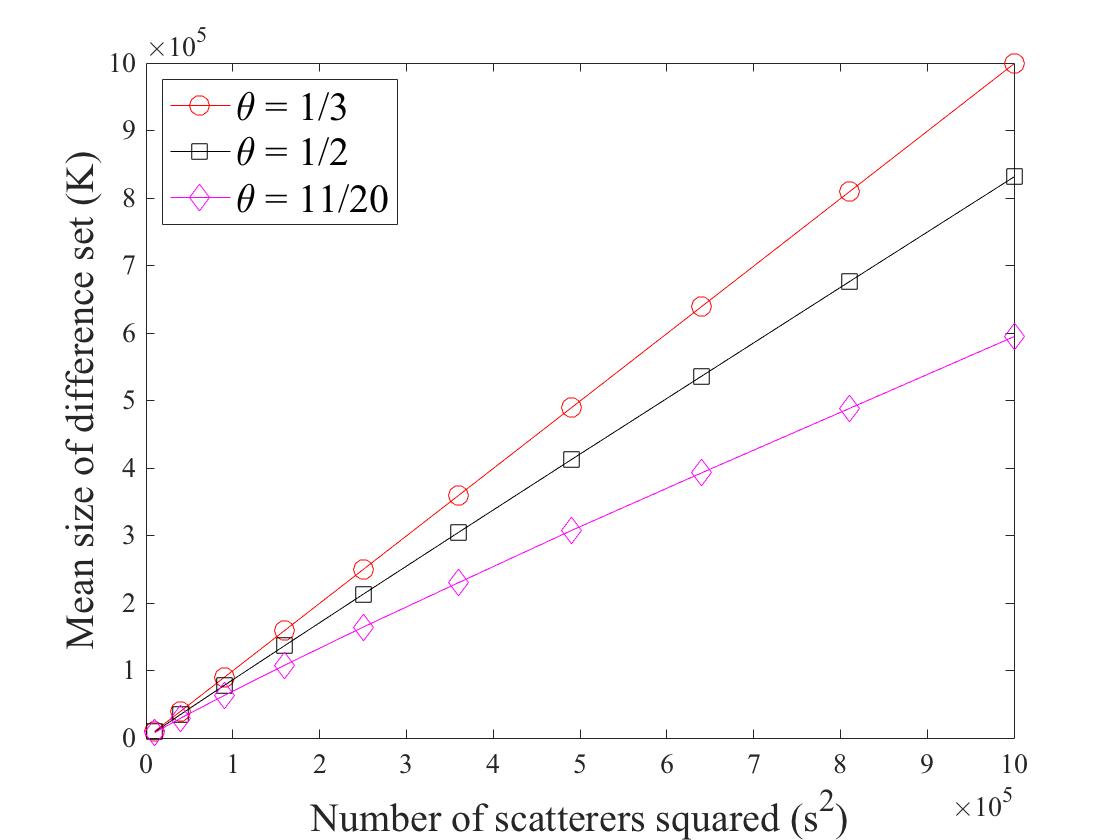}
  \caption{Size of difference set $W$ v. $s^2$}
    \label{sizeFig}
\end{subfigure}
\caption{\textbf{Recovery time.} Figures \ref{timeRelFigD3} show how average runtime increases with $s$ for different values of $\theta$. Figure \ref{timeSparsitySqrdFig} shows a nearly linear relationship between average runtime and $s^2$. Figure \ref{collabFig} shows the average ``collaboration depth'' (how deep in the tree $\mathcal{T}$ one must search before a solution is found) against $s$ for various values of $\theta$. Figure \ref{sizeFig} shows the average size $\kappa$ of the difference set $W$ against $s^2$, revealing a nearly linear relationship. All figures display simulations for $d=3$.}
\label{timeFig}
\end{figure*}

\begin{figure}
\centering
\includegraphics[width=.9\linewidth]{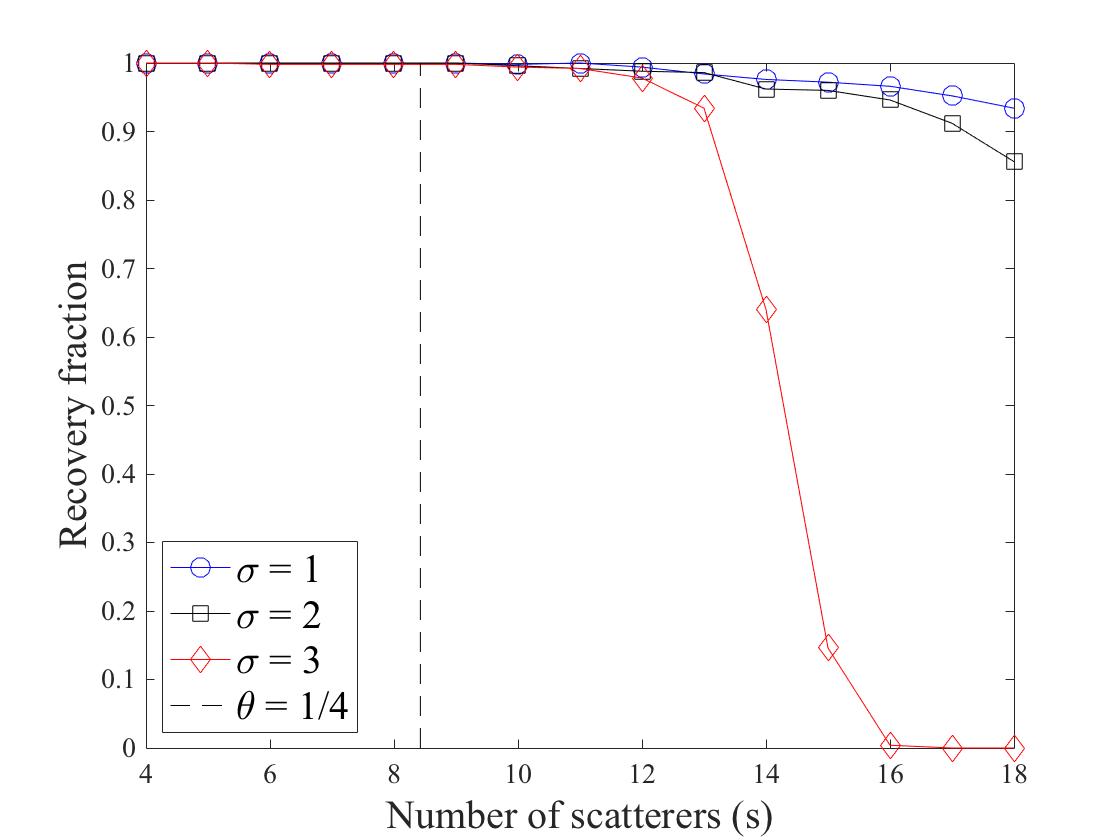}
\caption{\textbf{Noisy Phase Retrieval:} shows how exact support recovery probability varies with $s$ for several noise levels $\sigma = \|\bfa\|_2/\|\bfe\|_2$.}
\label{PRsupportFIG}
\end{figure}

The numerical simulations displayed in figures \ref{probFig} and \ref{timeFig} were implemented using MATLAB\textsuperscript{\textregistered} 2020a. The code used to generate these simulations is publicly available at \url{https://github.com/sew347/Sparse-Multidimensional-PR}. Data were simulated by the following process: $s$ was fixed as a parameter. $s$ points were generated from a $d$-dimensional normal distribution with mean 0 and covariance matrix $I$. These points were multiplied by $n/\sqrt{2\ln(s)}$, then discretized to $\mathbb{Z}^d$ rounded coordinate-wise to the nearest integer vector. Up to scaling, this distribution is the same as a normal distribution with mean $0$ and covariance $I$, normalized by $\sqrt{2\ln(s)}$ and discretized to $\mathbb{Z}_n^d$. Simulations for recovery probability were run for $1000$ iterations, while those for time and collaboration depth were run for 200 iterations. Simulations for the uniform distribution on a cube are included in the appendix.

Our implementation includes a minor optimization detailed in section \ref{random_vectors} to encourage the random vectors to be relatively decorrelated, which improves the chances of selecting different points in the convex hull as minima.
\subsection{Simulation Details}

\begin{enumerate}[wide, labelwidth=!, labelindent=0pt]
\item Our first simulation, displayed in figure \ref{probSparsityFig}, shows the recovery probability of MISTR when $\uptau = 30$, $d=3$ and $n^d = 10000$. Notably, this graph reflects high accuracy well above the theoretical accuracy cutoff of $\theta = 1/2$.

\item The second simulation, detailed in figure \ref{probProjFig}, shows the performance of MISTR with $s = 200$ in dimensions 2 through 4. Here $n$ is chosen such that $n^d = 40000$ for each $d$, so $\theta = 1/2$. We see that perfect recovery is obtained with sufficiently high $\uptau$ in every dimension, but fewer random projections are required to achieve the same recovery percentage. This corresponds to the greater number of vertices in $\conv{V}$ in higher dimensions.

\item The simulation displayed in figures \ref{timeRelFigD3} and \ref{timeSparsitySqrdFig} compares how the time complexity of MISTR evolves when sparsity grows as a fixed power of $n^d$ for $d = 3$. Here, $\uptau = 30$ while $n$ is chosen such that $s = n^{d\theta}$ for various values of $\theta$. These figures provide empirical evidence that, in practice, the time complexity of MISTR is $\bigO{\uptau s^2\log(s)} \approx \bigO{\uptau k^2\log{k}}$. This can be seen most clearly in \ref{timeSparsitySqrdFig}, which shows average runtime plotted against $s^2$; we can see that the resulting curves are nearly linear in $s^2$. We note that no recovery failures were observed for any of the plotted sparsity levels $\theta$.

\item Figure \ref{collabFig} shows the average ``collaboration depth'' for the same values of $\theta$ when $d=3$. Collaboration depth refers to the depth of the search tree $\mathcal{T}$ at which the first solution is found, including the root.  A collaboration depth of 1 implies that a solution was found during the first intersection step, so no collaboration search was required. Collaboration depth is a measure of how hard it is to find a solution for a given set of parameters, in the sense that higher collaboration depth means more information was needed before a solution was found. Accordingly, as $\theta$ grows, each intersection step retains more false positives, requiring a deeper collaboration search to filter them all out. This in turn results in a higher collaboration depth.

\item Figure \ref{sizeFig} compares the average size of the difference set $W$ to the squared number of scatterers $s^2$ over 100 samples. The graph shows a linear relationship for $\theta \leq 1/2$, though a slight nonlinearity for $\theta > 1/2$. Since $W$ is the input to our algorithm, the simulations displayed in figures \ref{sizeFig} and \ref{timeSparsitySqrdFig} corroborate our earlier assertion that, with high probability, MISTR runs in time nearly linear in the size of the input. This means that for $\theta \leq 1/2$, MISTR has close to the best asymptotic time complexity one can hope for in solving the combinatorial problem.

\item Figure \ref{PRsupportFIG} shows success probability for the full phase retrieval support recovery problem for several values of $s$ and noise levels $\sigma$ in dimension $d = 2$. The support of $\bfx$ was generated by the same process as in the noiseless case, but with values lying outside the square $[-2,2]\times[-2,2]$ truncated to maintain a uniform imaging window. The values of $\bfx$ were determined with a uniformly random phase and magnitude between 1 and 1.2, then normalized so $\|\bfa\|_2 = 1$. $\bfe$ was distributed as a normalized Gaussian with $\|\bfe\|_2 = \sigma$. The resolution parameter was fixed at $n = 71$ and the square was discretized into a $285 \times 285$ grid. 

The support of the resulting autocorrelation was estimated by thresholding with $\tau = 0.0114$. Failure was recorded if thresholding failed to exactly recover $\supp{\bfa}$ or if MISTR did not return an equivalent set to $\supp{\bfx}$. Performance was measured over 500 test runs. For all levels of noise tested, the results show good empirical recovery performance above the theoretical threshold. We see that for the lower noise levels $\sigma = 1$ and $\sigma = 2$, performance declines slowly, while for $\sigma = 3$ it drops sharply. This is because the errors for $\sigma = 1$ and $\sigma = 2$ are mainly caused by collisions in the support, while the sharper drop for $\sigma = 3$ because for larger $s$, $\sigma = 3$ is above the recoverable noise threshold even in the absence of collisions.

\end{enumerate}
\subsection{Discussion}\label{discussion}
Our numerical experiments verify the accuracy and showcase the speed of MISTR in practice. We observe by this experimental evidence that the typical runtime of MISTR is $\bigO{\uptau k^2\log(k)}$ as predicted in section \ref{3_algorithm}; this encourages our opinion that MISTR is especially well-suited for sparse problems when $n$ is very large. Our results confirm the predicted recovery properties, and even show successful recovery for values of $\theta$ substantially above the theoretically predicted cutoff of $\theta < 1/2$, even for large values of $s$. With parameters $s = 1000$, $\theta = 0.55$, and $\uptau = 30$, MISTR correctly recovered 1000 of 1000 test sets, despite $\theta = 0.55$ being significantly above the threshold for guaranteed recovery.

We are thus unsure whether the bounds in theorem \ref{gaussTheorem} are sharp. Perhaps an analysis that better harnesses properties of the Gaussian distribution might yield even better recovery guarantees. This will depend on how the required collaboration depth for recovery grows with large $n$ and $\theta > 1/2$: the number of different collaborations that MISTR can generate is approximately bounded by $|\Vertex{V}\hspace{-2pt}|$, and it is shown in \cite{BaVu:2007} that the number of points in the convex hull of $s$ $d$-dimensional Gaussian random variables concentrates around a mean of order $\bigO{\ln^{d/2}(s)}$ as $s$ grows large. Thus if the number of collaborations required for successful recovery grows faster than $\ln^{d/2}(s)$, this number will eventually exceed $|\Vertex{V}\hspace{-2pt}|$ and thus recovery by MISTR would be impossible. 

\subsection{Implementation details: random vector modification}\label{random_vectors}
To reduce the number $\uptau$ of random vectors required, we do not draw random vectors uniformly at random from $\mathbb{S}^{d-1}$ but rather encourage selecting vectors that are far apart in angle. This makes it more likely that each vector selects different points from $\Vertex{V}$, and so reduces the odds of ``wasted'' random vectors that do not add any new elements to $V_\uptau$. This procedure has no effect on the theoretical results above, but significantly reduces the time spent computing redundant projection steps. We detail the precise process in algorithm \ref{random_vector_selection}.

\begin{algorithm}[h]
\caption{Random Vector Selection} \label{random_vector_selection}
\textbf{Input:} Integer $\uptau$ \\
\textbf{Output:} Set $\mathcal{R}$ of $\uptau$ vectors $\mathcal{R} = \{\bfz^1, \bfz^2,\ldots, \bfz^\uptau\}$ 
\begin{algorithmic}[1]
\State $\mathcal{R} = \{\bfz^1\}$ where $\bfz^1$ is selected uniformly at random from $\mathbb{S}^{d-1}$.
\For{$i = 2:\uptau$}
    \State $r = 0$
    \While{TRUE}
        \State $r = r+1$
        \State $c = 1/(i + r)$
        \State Draw $\bfx$ uniformly at random from $\mathbb{S}^{d-1}$
        \State $\text{Corr}(X) = \max_{Y \in \mathcal{R}}|\inner{X}{Y}|$
        \If{$\text{{Corr}}(X) < 1-c$}
            \State \textbf{Break}
        \EndIf
    \EndWhile
    \State Add $\bfz^i = \bfx$ to $\mathcal{R}$.
\EndFor
\end{algorithmic}
\end{algorithm}

\subsection{Comparison with one-dimensional TSPR}
The algorithms TSPR and MISTR employ a similar strategy: they attempt to gather information on the support by using multiple intersection steps. As previously noted, one-dimensional TSPR resorts to a ``graph step'' which takes up to $\bigO{k^4}$ time. The details of this step mean TSPR is limited in how many intersection steps it can take without introducing false negatives: TSPR's graph step attempts to recover the first $p$ elements of a solution to use in $p$ additional intersection steps. However, there is a nonzero chance that this skips a support element (for example, it might recover $\tbfv_{1}, \tbfv_{2}$, and $\tbfv_{4}$, skipping $\tbfv_{3}$ and causing a false negative). To mitigate this risk, the number of intersections $p$ must be kept low---Jaganathan et al. prove asymptotic recovery for $p = 1+\sqrt[3]{\log (s)}$ intersections.

Because the structure of higher dimensions allows for nontrivial convex hulls, by taking different random projections MISTR can access multiple different geometrically compatible intersection steps. MISTR can thus access a significantly larger number of different intersection steps--- in 3 dimensions the number of points in the convex hull of $k$ Gaussian points grows as $\bigO{\ln^{3/2}(k)}$ \cite{BaVu:2007}---with no risk of false negatives and without relying on the $\bigO{k^4}$ graph step. 

\section{Conclusion}

In this work, we introduced MISTR to efficiently recover the support of a multidimensional signal from magnitude-only measurements. We proved that MISTR recovers most $\bigO{n^{d\theta}}$-sparse signals for $\theta < 1/2$ in the noiseless case, and for $\theta < 1/4$ in the presence of noise. Finally, numerical simulations verified these results and showed empirically that in practice MISTR runs in $\bigO{\uptau k^2\log(k)}$ time in the $\theta < 1/2$ regime. Remarkably, our empirical evidence shows that when data is distributed as Gaussian, MISTR is capable of recovery well above the theoretical threshold.

\bibliographystyle{IEEEtran} 
\bibliography{IEEEabrv,_refs}
\appendices
\renewcommand\thesubsectiondis{\Roman{subsection}.}
\section{Proofs}

\subsection{Proof of lemma \ref{all_points_lemma_Gauss}}

We begin with the proof of lemma \ref{all_points_lemma_Gauss}. For this, recall that for the random vectors $\bfz^i$, $i = 1,\ldots,\uptau$:
\[
\bfv_0^1 = \argmin_{\bfv \in V}\inner{\bfv}{\bfz^i} \hspace{0.1cm} \text{ and } \hspace{0.1cm} \bfv_1^1 = \argmin_{\bfv \in V\setminus \{\bfv_0^1\}}\inner{\bfv}{\bfz^i} 
\]
and 
\[
V_\uptau = \{\bfv_0^1,\bfv_1^1,\bfv_0^2,\bfv_1^2, \ldots, \bfv_0^\uptau, \bfv_1^\uptau\}
\]
Lastly, recall that we denote $\mfF_\bfl$ the event that $\{\bfl - \bfv_0^i, \bfl - \bfv_1^i\} \subseteq W$ for each $i = 1,2,\ldots, \uptau$.
\allPointsLemmaGauss*

The proof of this lemma is long, and so it makes sense to outline the general steps and approach involved. Step 1 is \textit{truncation}, where we bound the probability that points in $\Pi_s$ lie outside the ball of radius 2 using standard properties of the normal distribution along with a a well-known concentration inequality \cite{vershynin_2018}. Step 2 is \textit{coupling}, where we relate $\Pi_s$ and thereby $V_s$ to a homogeneous process on a discrete, bounded domain. Step 3 is a \textit{combinatorial bound}, which is proven in a pair of sub-lemmas \ref{homogeneous_lemma} and \ref{asymptotic_prob_all_points}. These two lemmas are essentially lemmas VIII.4 and VIII.5 from \cite{JaOyHa:2017}, adapted to our multidimensional setting; proofs are provided in the following section.

\begin{proof}
Recall that $\mfF_\bfl$ is the event that $\{\bfl - \bfv_0^i, \bfl - \bfv_1^i\} \subseteq W$ for all $i = 1,2, \ldots, \uptau$. Conditioned on $\bfl \notin V_s$, this is the event that $\bfl \in U$, the event that MISTR returns $\bfl$ as a false positive inclusion. Denote the event:
\[
\mfF = \bigcup_{\bfl \in \mathbb{Z}_n^d} \mfF_\bfl \cap \{\bfl \notin V_s\}
\]
$\mfF$ is the event that there exists an $\bfl \in \mathbb{Z}_n^d$ which is returned by MISTR as a false positive. As previously noted, this is exactly the event in which MISTR fails, so our goal in the following is to bound $\prob(\mfF)$.

Step 1: \textit{truncation}. We begin by relating the process $V_s$ to a process on a bounded domain. First, consider the conditional distribution $(\Pi_s|k) = \Pi_s \mid (|\Pi_s| = k)$. It is clear that this is distributed as $k$ independent Gaussian random variables distributed by $\Phi_s$, while $k = |\Pi_s|$ is a Poisson random variable with mean $s$. It follows that the process $\frac{\sqrt{2\ln(s)}}{\sqrt{2\ln(k)}}(\Pi_s|k)$ is distributed as $k$ i.i.d. Gaussian random variables with mean $0$ and covariance matrix $\frac{1}{\sqrt{\ln(k)}}I$. Denote this distribution $\Pi^k$.

It is known that for $k$ independent, identically distributed Gaussian random vectors $\{\bfz^i\}_{i=1}^k$ with mean $0$ and covariance matrix $I$, $E\max_{i}\|\bfz^i\|_2$ behaves like $\sqrt{2\ln(k)}$ for $k$ sufficiently large. By concentration of Lipschitz functions of Gaussian random variables, see e.g. \cite{vershynin_2018}, this maximum enjoys sub-Gaussian concentration. In particular, we have that:
\begin{multline*}
\prob\left(\max_{i=1,\ldots,k}\|\bfz^i\|_2 -E\max_{i=1,\ldots,k}\|\bfz^i\|_2 \geq \sqrt{2\ln(k)}\right) \leq e^{-\frac{\ln{k}}{C}}
\end{multline*}
for $C$ an absolute positive constant. Thus:
\begin{multline*}
\prob\left(\max_{p\in \Pi^k}\|p\|_2 \geq 2\right) = \prob\left(\max_{i=1,\ldots,k}\frac{\|\bfz^i\|_2}{\sqrt{2\ln(k)}} \geq 2\right) \leq \frac{1}{k^{1/C}}
\end{multline*}
Since $k$ is distributed as a Poisson random variable with mean $s$,$\frac{\sqrt{2\ln(s)}}{\sqrt{2\ln(k)}} \to 1$ almost surely by the law of large numbers. It follows that $\Pi^k \to (\Pi_s|k)$ almost surely in the Hausdorff metric. This plus the subexponential concentration of Poisson random variables guarantees that for $n$ large enough,
$\prob\left(\max_{p\in \Pi_s}\|p\| \geq 2\right) \leq \frac{2}{s^{1/C}}$.

Let $\Omega_s$ be the event that all points in $\Pi_s$ fall inside the ball of radius 2. By conditional probability calculations, we have:
\begin{multline}\label{truncation_relation}
\prob(\mfF) = \prob(\mfF | \Omega_s)\prob(\Omega_s) + \prob(\mfF \mid \Omega_s^c)\prob(\Omega_s^c) \leq \\ \leq \prob(\mfF | \Omega_s) + \prob(\Omega_s^c) \leq \prob(\mfF | \Omega_s) + \frac{2}{s^{1/C}}
\end{multline}
It thus suffices to show that $\prob(\mfF | \Omega_s)$ obeys the desired bounds. 

Step 2: \textit{Coupling}. By the independence property of the Poisson process, the conditional process $\left(\Pi_s \mid \Omega_s\right)$ will be distributed as $\Pi_s$ on the ball of radius 2, and contain no points on or outside of this ball. We now note that since the normalized density $\Phi_s$ is bounded by $\ln^{d/2}(s)$, we can bound  $\left(\Pi_s \mid \Omega_s\right)$ by a \textit{homogeneous} Poisson point process $\pi_s$ on $B(0,2)$ with uniform intensity $s\ln^{d/2}(s)$. Here the bound is in the sense that for any set $A \subseteq B(0,2)$ and any $j$,
\begin{equation}\label{process_bound}
\prob(|A \cap \Pi_s| \geq j \mid \Omega_s) \leq \prob(|A \cap \pi_s| \geq j)
\end{equation}
Denote $\tZnd$ the set of all $z \in \mathbb{Z}_n^d$ such that that $B(0,2) \cap P_n(\bfc) \neq \emptyset$. We can then define $Y_s$ be the discretization of $\pi_s$ on $\tZnd$, where a point $\bfc \in P_{R_N}$ is in $Y_s$ iff $|P_n(\bfc) \cap \pi_s| \geq 1$. Since $\pi_s$ is homogeneous with intensity $s$, it follows that every point $z \in P_{R_n}$ is in $Y_s$ with independent probability at most $1-\exp\left(-\frac{s\ln^{d/2}(s)}{n^d}\right) = \bigO{\frac{s\ln^{d/2}(s)}{n^d}}$. Denote this by:
\begin{equation}\label{rho_n}
\rho_n = 1-\exp\left(-\frac{s\ln^{d/2}(s)}{n^d}\right)
\end{equation}
It follows that $\rho_n \to 0$ as $n \to \infty$ as long as $\theta < 1$.

By the above definitions, we see that the discrete, homogeneous process $Y_s$ bounds the discretized Poisson process $V_s$ in the sense of (\ref{process_bound}). Thus we know that for any $\bfc \in P_{R_N}$
\begin{equation}
    \prob(z \in V_s) \leq \prob(z \in Y_s) = \prob(|\pi_s \cap P_n(\bfc)| \geq 1) \leq \rho_n
\end{equation}
In other words, each $z$ in $\mathbb{Z}_n^d \cap B(0,2)$ is in $V_s$ independently with probability not exceeding $\rho_n$.

The following lemmas comprise step 3: the \textit{combinatorial bound}. Due to their length, full proofs are relegated to our supplementary materials. Specifically, we bound the probability that a pixel $\bfl$ appears in $\bigcap_{i=1}^\uptau \Theta_i(U(i))$ given $\bfl \notin V_s$:

\begin{restatable}{lemma}{homogeneousLemma}
\label{homogeneous_lemma}
Let $\uptau \in \mathbb{N}$ and $s = n^{d\theta}$ for $d\geq 2$, $\theta \in (0,1/2)$. Let $\{\bfz^j\}_{j=1}^\uptau$ be independent, uniformly distributed random vectors on $\mathbb{S}^{d-1}$. Let $V_\uptau$ be as in \ref{vuptau}. Then  if $|\dist{V_\uptau}| = t \in \mathbb{N}$, for any $\bfl \in \tZnd$ the probability that $\bfl \notin V_s$ and that
\begin{equation}\label{ellEvent}
\{\bfl - \bfv_0(i), \bfl - \bfv_1(i)\} \subseteq W \text{ for each } i = 1,2,\ldots, \uptau
\end{equation}
is bounded by $\bigO{\frac{s^{\sqrt{t}/2}\ln^{d\sqrt{t}/4}(s)}{n^{d\sqrt{t}/4}}}$ as $n \to \infty$.
\end{restatable}

This result is proven through combinatorial reasoning, counting the various ways that a false positive $\bfl$ could persist through the intersection and collaboration steps. The next lemma expands this result to apply to all $\bfl$ simultaneously:

\begin{restatable}{lemma}{asymptoticProbAllPoints}
\label{asymptotic_prob_all_points}
Under the same assumptions as in lemma \ref{homogeneous_lemma}, the probability that there exists an $\bfl \in \tZnd$, such that $\bfl \notin V_s$ and $\mfF_{\bfl}$ occurs is of order at most:
\[
\bigO{\frac{s^{\sqrt{t}/2}\ln^{d\sqrt{t}/4}(s)}{n^{d(\sqrt{t}/4-1)}}} = \bigO{\frac{\ln^{d\sqrt{t}/4}(n)}{n^{d\left(\frac{(1/2-\theta) \sqrt{t}}{2} - 1\right)}}}
\]
\end{restatable}

The proof of this lemma involves a straightforward application of the Markov inequality. This lemma proves the order bound for the conditional Poisson process $\left(\Pi_s \mid \Omega_s\right)$ on $\tZnd$, which with (\ref{truncation_relation}) completes the proof of lemma \ref{all_points_lemma_Gauss}.
\end{proof}
\subsection{Proof of lemma \ref{twoPointLemma}}
We can use the same truncation and coupling steps from the previous section to prove lemma \ref{twoPointLemma}:
\twoPointLemma*
\begin{proof}
Consider the same homogeneous process $\pi_s$ on $B(0,2)$ with intensity $s$ as in the proof of the previous lemma. By the same arguments, it suffices to prove the order bound for $\pi_s$. For any $\bfc \in \tZnd$, $\prob{\left(|\pi_s \cap P_n(\bfc)| \geq 2\right)}$ is at most:
\begin{multline*}
\prob(|P_n(\bfc) \cap \pi_s| \geq 2) \leq 1 - \exp\left(\frac{-s\ln^{d/2}(s)}{n^d}\right) -\\- \frac{s\ln^{d/2}(s)}{n^d}\exp\left(\frac{-s\ln^{d/2}(s)}{n^d}\right)= \bigO{\frac{s^2\ln^{d}(s)}{n^{2d}}}
\end{multline*}
As there are $\bigO{n^d}$ pixels in $\tZnd$, we conclude from a union bound that:
\begin{multline*}
\prob(\exists \bfc \in \tZnd \text{ s.t. } |P_n(\bfa) \cap \pi_s| \geq 2) \leq\\\leq \bigO{n^d \times \frac{s^2\ln^{d}(s)}{n^{2d}}} = \bigO{\frac{s^2\ln^{d}(s)}{n^{d}}}
\end{multline*}
\end{proof}
We conclude with the proof of lemma \ref{distProbLemma}, which completes the proof of theorem \ref{gaussTheorem}.
\subsection{Proof of lemma \ref{distProbLemma}}
\distProbLemma*
The idea of the proof this: we will show that the number of distinct elements in the subset $\{\bfv_0^1,\bfv_0^2,\ldots,\bfv_0^\uptau\}$ converges to $\uptau$ in probability. Then, as $n$ gets large, the number of distinct elements in $V_\uptau$ must be $\uptau$ or greater with probability tending to $1$. We use the fact that $\bfv_0^i$ must be an element in $\Vertex{V}$ (in particular, it is a minimal point); consequently, it suffices to prove that the random vectors $\bfz^i$ each select a different minimal point with probability tending to one as $n \to \infty$.

Specifically, we recall the relationship between $\Pi_s$ and the process $\Pi^k$ consisting of a fixed number of independent, normalized Gaussian random vectors. We then apply corollary \ref{corollaryDavydov} to conclude that $\Pi^k$ converges almost surely to the unit ball in the Hausdorff metric. It follows that $\Vertex{\Pi^k}$ converges in probability to $\mathbb{S}^{d-1}$ in the Hausdorff metric; thus these vertices collect roughly uniformly on the sphere $\mathbb{S}^{d-1}$. Therefore, the probability that the same vertex has the maximal inner product with more than one of a finite number $\uptau$ of random unit vectors tends to zero as $n \to \infty$. By the relationship between $\Pi_s$ and $\Pi^k$, this also holds for $\Vertex{\Pi_s}$, and the result follows.

\begin{proof}

Recall the renormalized Poisson process $(\Pi_s|k) = (\Pi_s \mid |\Pi_s| = k)$, and let $(V_s|k) = (V_s \mid |\Pi_s| = k)$ and $(\mfD_s|k) = (|\dist{\{\bfv_0^1,\bfv_0^2,\ldots,\bfv_0^\uptau\}}| \mid |\Pi_s| = k)$. By definition of the processes $\Pi_s$ and $V_s$, $k$ is distributed as a Poisson random variable with mean $s = n^{d\theta}$, so $k \to \infty$ as $s \to \infty$. Thus if $(\mfD_s|k) \to \uptau$ in probability as $k \to \infty$, it follows that $|\dist{\{\bfv_0^1,\bfv_0^2,\ldots,\bfv_0^\uptau\}}| \to \uptau$ in probability.

We recall further that $\Pi^k = \frac{\sqrt{2\ln(s)}}{\sqrt{2\ln(k)}}(\Pi_s|k)$ is distributed as $k$ i.i.d. Gaussian random variables with covariance $\frac{1}{\sqrt{2\ln(k)}}I$, so we can apply corollary \ref{corollaryDavydov} to conclude that almost surely, $\conv{\Pi^k}$ converges to $B(0,1)$. We already saw that $\Pi^k$ converges to $(\Pi_s|k)$ almost surely in the Hausdorff metric, so $\conv{\Pi_s|k}$ also converges to $B(0,1)$ almost surely.

As $V_s$ is the discretization of the Poisson process $\Pi_s$ on the grid $\mathbb{Z}_n^d$, as $n \to \infty$, $V_s$ becomes arbitrarily close to $\Pi_s$ in the Hausdorff metric. It follows that as $k \to \infty$, $(V_s|k)$ also converges to $B(0,1)$ almost surely in this metric. In particular, since the boundary of $B(0,1)$ has positive curvature at every point, we can conclude that the vertices $\Vertex{V_s|k} \to \mathbb{S}^{d-1}$ almost surely as $k\to\infty$.

Now, for any $\eta$ we can choose a set $\mathcal{S}_\eta$ of finitely many points from $\mathbb{S}^{d-1}$ such that for a random vector $X$ chosen uniformly from $\mathbb{S}^{d-1}$,
\[
\max_{s^* \in \mathcal{S}_\eta}\prob\left( \argmin_{s \in \mathcal{S}} \inner{X}{s} = s^* \right) \leq \frac{1}{\eta}
\]

Since $\Vertex{V_s|k}$ converges to $\mathbb{S}^{d-1}$, by continuity of projections there exists $k$ large enough that we can guarantee that given two unit vectors $\bfz^1 \neq \bfz^2$, $\argmin_{s\in {\mathcal{S}_\eta}} \inner{\bfz^1}{s} \neq \argmin_{s\in {\mathcal{S}}_\eta} \inner{\bfz^2}{s}$ implies that $\argmin_{\bfv\in (V_s|k)} \inner{\bfz^1}{\bfv} \neq \argmin_{\bfv\in (V_s|k)} \inner{\bfz^2}{\bfv}$. It follows that for any $\eta$, there exists $k_0$ such that for all $k > k_0$, for any $X$ distributed uniformly at random on $\mathbb{S}^{d-1}$:
\begin{equation}\label{maxBound}
\max_{a \in \Vertex{(V_s|k)}} \prob\left(\argmin_{\bfv \in (V_s|k)}\inner{X}{\bfv} = a \hspace{2pt} \right) \leq \frac{1}{\eta(k)}
\end{equation}
where $\eta(k) \to \infty$ as $k \to \infty$. 

We now show that this is sufficient to conclude that $(\mfD_s|k)\to \uptau$ in probability. In order for $(\mfD_s|k)$ to be less than $\uptau$, there must be two random vectors $\bfz^i$ and $\bfz^j$, $i \neq j$, which select the same minimal point. We can thus write:
\[
\prob((\mfD_s|k) < \uptau) \leq \sum_{i\neq j}^\uptau \prob\left(\argmin_{\bfv \in (V_s|k)}\inner{\bfz^i}{\bfv} = \argmin_{\bfv \in (V_s|k)}\inner{\bfz^j}{\bfv} \hspace{2pt}\right)
\]
Since the $\bfz^i$'s are independent and identically distributed, by a union bound:
\[
\prob((\mfD_s|k) < \uptau) \leq {\uptau \choose 2}\prob\left(\argmin_{\bfv \in (V_s|k)}\inner{\bfz^1}{\bfv} = \argmin_{\bfv \in (V_s|k)}\inner{\bfz^2}{\bfv} \hspace{2pt} \right)
\]
Thus it suffices to show that 
\[\prob\left(\argmin_{\bfv \in (V_s|k)}\inner{\bfz^1}{\bfv} = \argmin_{\bfv \in (V_s|k)}\inner{\bfz^2}{\bfv}\right) \to 0 \text{ as } k \to \infty
\]
Denote $\mathcal{A}_a$ the event that $\argmin_{\bfv \in (V_s|k)}\inner{\bfz^1}{\bfv} = a$ for $a\in (V_s|k)$. Conditioning on $\mathcal{A}_a$ we have:
\begin{multline}
\prob\left(\argmin_{\bfv \in (V_s|k)}\inner{\bfz^1}{\bfv} = \argmin_{\bfv \in (V_s|k)}\inner{\bfz^2}{\bfv} \hspace{2pt} \right) = \\ =\sum_{a \in \Vertex{(V_s|k)}} \prob\left(\argmin_{\bfv \in (V_s|k)}\inner{\bfz^2}{\bfv} = a \hspace{2pt} \bigg\lvert \hspace{2pt} \mathcal{A}_a\right)\prob\left(\mathcal{A}_a\right) \leq \\ 
\leq \max_{a \in \Vertex{(V_s|k)}} \prob\left(\argmin_{\bfv \in (V_s|k)}\inner{\bfz^2}{\bfv} = a \hspace{2pt} \right) \leq \frac{1}{\eta(k)}
\end{multline}
by (\ref{maxBound}). Since $\eta(k) \to \infty$ as $k \to \infty$, this converges to $0$. This proves that $(\mfD_s|k) \to \uptau$ in probability, which as previously noted is sufficient to conclude that $|\dist{V_\uptau}| \geq \uptau$ with probability tending to one as $n \to \infty$.
\end{proof}

We conclude with a comment on our assumption $\tV^i = V_s - \bfv_0^i$. Without this assumption, for each $\bfz^i$, $V_\uptau$ contains either $\bfv_0^i = \argmin_{\bfv \in V_s} \inner{\bfz^i}{\bfv}$ or $\bfv_{k-1}^i = \argmax_{\bfv \in V_s} \inner{\bfz^i}{\bfv}$. Using the same argument used in the above lemma, as $n \to \infty$, the $2\uptau$ elements $\{\bfv_0^1,\bfv_{k-1}^1,\ldots, \bfv_0^\uptau,\bfv_{k-1}^\uptau\}$ will be distinct with high probability. Thus, regardless of whether the max or the min is included in $V_\uptau$ for each $i$, it is guaranteed that $|\dist{V_\uptau}| \geq \uptau$ with high probability as $n \to \infty$.

\subsection{Proofs of Technical Lemmas}
This section contains proofs of lemmas \ref{homogeneous_lemma} and \ref{asymptotic_prob_all_points}, which are adaptations of lemmas VIII.4 and VIII.5 from \cite{JaOyHa:2017}.
\homogeneousLemma*
\begin{proof}
Since all probabilities in this lemma are conditioned on $|\dist{V_\uptau}| = t$, we do not write this explicitly. 

Recall that $\tZnd := \mathbb{Z}_n^d \cap B(0,2)$. For fixed dimension $d$, there exists a constant depending only on dimension $c_d$ such that $\tZnd$ contains $K_n \leq c_d 2^dn^d = \bigO{n^d}$ points. Recall that from our earlier reasoning, each point in $\tZnd$ is in $V_s$ with probability at most $\rho_n = 1 - \exp\left(\frac{-s\ln^{d/2}(s)}{n^d}\right)=\bigO{\frac{sln^{d/2}(s)}{n^d}}$.

Recall that $\mfF_\bfl$ denotes the event $\{\bfl - \bfv_0(i), \bfl - \bfv_1(i)\} \subseteq W$ for each $i$. We begin by conditioning on $\dist{V_\uptau} = Y$, where $Y = (\bfy_1,\ldots,\bfy_t)$ ranges over all combinations of $t$ distinct elements in $\tZnd$:

\begin{equation}\label{conditioning_on_d}
    \prob(\mfF_\bfl) = \sum_{Y} \prob(\mfF_\bfl \mid V_\uptau = Y)\prob(V_\uptau = Y)
\end{equation}

Our goal is to show that the probability is small that all the events $(\bfl - \bfy_i) \in W^i$ occur simultaneously. Our strategy follows three steps: first, we prove that interpoint differences among points in $Y$ will be unique with high probability as $n\to \infty$. Then, by a counting argument, we show that when differences between elements in $Y$ are unique, $\prob(\mfF_\bfl \mid V_\uptau = Y)$ is small with high probability as $n \to \infty$. We then conclude the result with some conditional probability computations.

First, we note that interpoint differences among points in $Y$ will be unique with probability arbitrarily close to 1 for large enough $n$. We bound this probability as follows. Denote the event $\mcU_\eta$ the event in which the points $\bfy_1,\ldots, \bfy_\eta$ have unique pairwise differences. Then:
\[
\prob(\mcU_t) = \prod_{\eta=1}^t \prob(\mcU_\eta | \mcU_{\eta-1})
\]
We now bound $\prob(\mcU_\eta | \mcU_{\eta-1})$. By conditioning on $\mcU_{\eta-1}$, we know that pairwise differences not involving $\bfy_\eta$ will be unique. For a fixed choice of $\bfy_{i_\eta},\bfy_{j_\eta},\bfy_{h_\eta}$ with each index less than $\eta$, we have that $\bfy_{i_\eta} - \bfy_{j_\eta} = \bfy_{h_\eta} - \bfy_\eta \iff \bfy_{\eta} = \bfy_{j_\eta} + \bfy_{h_\eta} - \bfy_{i_\eta}$. This probability is upper bounded by $\rho_n$ as in (\ref{rho_n}). Unfixing $i_\eta,j_\eta,h_\eta$, there are at most $(\eta-1)^3$ choices for $i_\eta, j_\eta$, and $h_\eta$, so the probability that $\mcU_\eta \mid \mcU_{\eta-1}$ fails is upper bounded by $(\eta-1)^3\rho_n$. It follows that:
\begin{multline}
\prob(\mcU_t) = \prod_{\eta=1}^t \prob(\mcU_\eta | \mcU_{\eta-1}) \geq \prod_{\eta=3}^t 1 - (\eta-1)^3\rho_n \geq\\\geq 1 - t^4\rho_n + \bigO{\rho_n^2}\label{distinctProb}
\end{multline}
Since $t$ is finite and $\rho_n \to 0$, $\prob(\mcU_t)$ converges to 1 as $n \to \infty$. Thus we conclude that the pairwise differences in $Y$ are unique with high probability for $n$ large. Denote by $\mathcal{D}$ the set of all possible $Y$ that have unique pairwise differences. We can now bound (\ref{conditioning_on_d}) by:

\begin{equation}\label{conditioning_on_unique_d}
    \prob(\mfF_\bfl) \leq \bigO{\rho_n} + \sum_{Y \in \mathcal{D}} \prob(\mfF_\bfl \mid V_\uptau = Y)\prob(V_\uptau = Y)
\end{equation}

From this it is clear that it suffices to prove the result for $Y \in \mathcal{D}$. We now proceed with the counting argument to bound $\prob(\mfF_\bfl \mid V_\uptau = Y)$ for $Y \in \mathcal{D}$. We say that a pair of points $\bfv_1,\bfv_2 \in V_s$ is an \textit{explaining pair} for a difference $\bfl - \bfy_i$ if $\bfv_1 - \bfv_2 = \bfl - \bfy_i$. We will also write that a point $v$ \textit{explains} a difference to mean that $v$ is a member of an explaining pair for that difference.

When $Y \in \mathcal{D}$, the points in $Y$ can explain at most $t/2$ differences by the interpoint differences among themselves. To see this, suppose that $\bfy_i -\bfy_{j_1} = \bfl - \bfy_i - \bfy_{k_1}$ and $\bfy_i -\bfy_{j_2} = \bfl - \bfy_i - \bfy_{k_2}$. By adding these, $\bfy_{j_2} -\bfy_{j_1} = \bfy_{k_2} - \bfy_{k_1}$, which happens with arbitrarily low probability for large enough $n$. Since there are at most $t/2$ disjoint pairs of elements in $Y$, it follows that pairs of points in $Y$ can be explaining pairs for at most $t/2$ differences.

Thus at least $t/2$ differences must be explained by other points in $V_s$. This can only happen due to one of the following cases:

\begin{enumerate}[wide, labelindent=0pt, label=(\roman*)]

\item There exists $\bfg \in V_s$ that forms at least two explaining pairs with different elements in $Y$. Suppose that $\bfg - \bfy_{i_1} = \bfl - \bfy_{j_1}$ and $\bfg - \bfy_{i_2} = \bfl - \bfy_{j_2}$. Adding these, this implies $\bfy_{i_1} - \bfy_{i_2} = \bfy_{j_1} - \bfy_{j_2}$, which contradicts $Y \in \mathcal{D}$ unless $\bfg = \bfl$. Thus the probability of this case is $\prob(\bfl \in V_s)$.

Let $\mathcal{G}$ be the set of points in $\tZnd$ that can form an explaining pair of a difference $\bfl - \bfy_i$ using only points from $Y$. Each $\bfg$ in $\mathcal{G}$ can be form an explaining pair with at most one point in $Y$, as otherwise this would put us in case (i). For a point $\bfg$ to be in this set, it must satisfy $\bfg-\bfy_i = \bfl - \bfy_j$ for some $i,j$. As there are at most $t$ choices each for $\bfy_i$ and $\bfy_j$, $\mathcal{G}$ is composed of at most $t^2$ points.

\item We consider the case when $\eta \geq t/4$ differences are explained by explaining pairs with one point in $\mathcal{G}$ and one point in $Y$. For each such difference, a different $\bfg \in \mathcal{G}$ must be chosen. Since there are at most $t^2$ points in $\mathcal{G}$, there are at most $t^{2\eta}$ ways to choose $\eta$ points in $\mathcal{G}$, each of which will be in $V_s$ with probability at most $\rho_n$. Thus the probability of this case is bounded by:
\[
\sum_{\eta = t/4}^{t} t^{2\eta} (\rho_n)^\eta \leq t^{2t}(\rho_n)^{t/4}
\]
\item At least $t/4$ differences are explained by pairs of points not involving $Y$. There are two possibilities:
\begin{enumerate}[wide, labelwidth=!, labelindent=0pt]
\item There exist at least $t/4$ points $\{\bfg_1,\bfg_2,\ldots, \bfg_{t/4}\}$ in $V_s$ such that $\{\bfg_1, \bfg_1 + \bfl - \bfy_{p_1},\ldots \bfg_{t/4}, \bfg_{t/4} + \bfl - \bfy_{p_{t/4}}\}$ are all in $V_s$ and are distinct. Given $\bfg_i \in V_s$, the probability that $\bfg_i + \bfl - \bfy_{p_i}$ is in $V_s$ is bounded by $\rho_n$. Unfixing $\bfg_i$ by a union bound, as there are $K_n$ possible choices for $\bfg_i$ in $\tZnd$, the probability that there exists $\bfg_i$ and $\bfg_i + \bfl - \bfy_{p_i}$ both in $V_s$ is upper bounded by $K_n\rho_n$. By independence, the probability that this case occurs is at most $(K_n\rho_n)^{t/4}$.

\item \label{last_possibility} There exist at least $t/4$ points $\{\bfg_1,\bfg_2,\ldots, \bfg_{t/4}\}$ in $V_s$ such that $\{\bfg_1, \bfl - \bfy_{p_1},\ldots \bfg_{t/4}, \bfg_{t/4} + \bfl - \bfy_{p_{t/4}}\}$ are all in $V_s$, but are not distinct. Let $\eta$ be the number of distinct points in the latter set. Since these points need to generate at least $t/4$ distinct differences, we must have that $\eta \geq \sqrt{t}/2$. Jaganathan et al. compute in \cite{JaOyHa:2017} that there are at most $5^{t^2}$ ways these points can overlap, and at most $\eta/2$ points that are not determined by their overlap with others. Since these can each be chosen each in $K_n$ ways, we have that the probability of this case is at most:
\[
\sum_{\eta = \sqrt{t}/2}^{2t} 5^{t^2} (K_n)^{\eta/2} (\rho_n)^\eta \leq 2t(5^{t^2})((K_n)^{1/2}\rho_n)^{\sqrt{t}/2}
\]
\end{enumerate}
\end{enumerate}
Recalling that $K_n = \bigO{n^d}$ and $\rho_n = \bigO{\frac{s}{n^d}}$, we see that case \ref{last_possibility} gives the weakest asymptotic bound in terms of $s,n,$ and $t$. It follows that the probability that $\mfF_\bfl$ occurs, given $V_\uptau = Y$, is bounded by:
\begin{multline*}
\prob(\mfF_\bfl \mid V_\uptau = Y) \leq \prob(\bfl \in V_s) + \bigO{\left(\frac{s\ln^{d/2}(s)}{n^{d/2}}\right)^{\sqrt{t}/2}} =\\= \prob(\bfl \in V_s) + \bigO{\left(\frac{s^2}{n^{d}}\right)^{\sqrt{t}/4}\ln^{d\sqrt{t}/4}(s)}
\end{multline*}
Then from (\ref{conditioning_on_unique_d}), we have:
\begin{multline}\label{probability_ell_in_V_s}
    \prob(\mfF_\bfl) \leq \bigO{\rho_n} + \sum_{Y \in \mathcal{D}} \prob(V_\uptau = Y) \times \\\times \left(\prob(\bfl \in V_s) + \bigO{\left(\frac{s^2}{n^{d}}\right)^{\sqrt{t}/4}\ln^{d\sqrt{t}/4}(s)}\right) \leq \\ \leq
        \bigO{\frac{s\ln^{d/2}(s)}{n^d}} + \left(\vphantom{\bigO{\left(\frac{s^2}{n^{d}}\right)^{\sqrt{t}/4}}}\prob(\bfl \in V_s)\right. + \\ + \left.\bigO{\left(\frac{s^2}{n^{d}}\right)^{\sqrt{t}/4}\ln^{d\sqrt{t}/4}(s)}\right) \times \sum_{Y \in \mathcal{D}} \prob(V_\uptau = Y) \leq \\ 
        \leq  \prob(\bfl \in V_s) + \bigO{\left(\frac{s^2}{n^{d}}\right)^{\sqrt{t}/4}\ln^{d\sqrt{t}/4}(s)}
\end{multline}
since $\sum_{Y \in \mathcal{D}} \prob(V_\uptau = Y) \leq 1$.

We are almost finished. Now, note that we can write
\[
    \prob(\mfF_\bfl) = \prob(\mfF_\bfl \mid \bfl \in V_s)\prob(\bfl \in V_s) + \prob(\mfF_\bfl \mid \bfl \notin V_s))\prob(\bfl \notin V_s))
\]
Since $\prob(\mfF_\bfl \mid \bfl \in V_s) = 1$, we have:
\[
\prob(\mfF_\bfl \mid \bfl \notin V_s) = \frac{\prob(\mfF_\bfl) - \prob(\bfl \in V_s)}{\prob(\bfl \notin V_s)}
\]
Thus by plugging in (\ref{probability_ell_in_V_s}), and noting that $\prob(\bfl \notin V_s) \to 1$ as $n \to \infty$, we conclude:
\[
\prob(\mfF_\bfl \mid \bfl \notin V_s)) \leq \bigO{\left(\frac{s^2}{n^{d}}\right)^{\sqrt{t}/4}\ln^{d\sqrt{t}/4}(s)}
\]
This completes the proof of lemma \ref{homogeneous_lemma}.
\end{proof}

\asymptoticProbAllPoints*
\begin{proof}
Let $L$ be a random variable defined as the number of false positives; that is, the number of $\bfl$ such that $\bfl \notin V_s$ but $\mfF_\bfl$ occurs. Under this definition, we need to bound $\prob(L \geq 1)$, which we do as follows:
\begin{multline*}
E[L] \leq \sum_{\bfl \in \tZnd} \prob(\mfF_\bfl \mid \bfl \notin V_s)\leq\\ \leq K_n \times \bigO{\left(\frac{S^{\sqrt{t}/2}}{n^{d\sqrt{t}/4}}\right)\ln^{d\sqrt{t}/4}(s) }
\end{multline*}
By the Markov inequality and substituting $K_n =  \bigO{n^d}$  it follows that:
\begin{multline*}
\prob(L \geq 1) \leq \bigO{K_n \times \left(\frac{S^{\sqrt{t}/2}}{n^{d\sqrt{t}/4}}\right)\ln^{d\sqrt{t}/4}(s)} =\\= \bigO{\left(\frac{s^{\sqrt{t}/2}}{n^{d(\sqrt{t}/4-1)}}\right)\ln^{d\sqrt{t}/4}(s)}
\end{multline*}
Expanding $s$ in terms of $n$ gives the result.
\end{proof}

\section{Supplementary Simulations}
In this section we provide additional numerical simulations for the case when $V_s$ is distributed as $s$ elements chosen uniformly from an integer grid $\{0,1,\ldots,n-1\}^d$. We show that this distribution enjoys essentially the same performance as the Gaussian distribution featured in our paper as long as $\theta < 1/2$. Above this threshold, good recovery performance is still possible, but time complexity rises. By seeing similar performance on another distribution, this reinforces our prior evidence that MISTR is highly scalable as long as sparsity remains below the $\theta < 1/2$ threshold.

\begin{figure}[h]
  \centering
  \includegraphics[width=.9\linewidth]{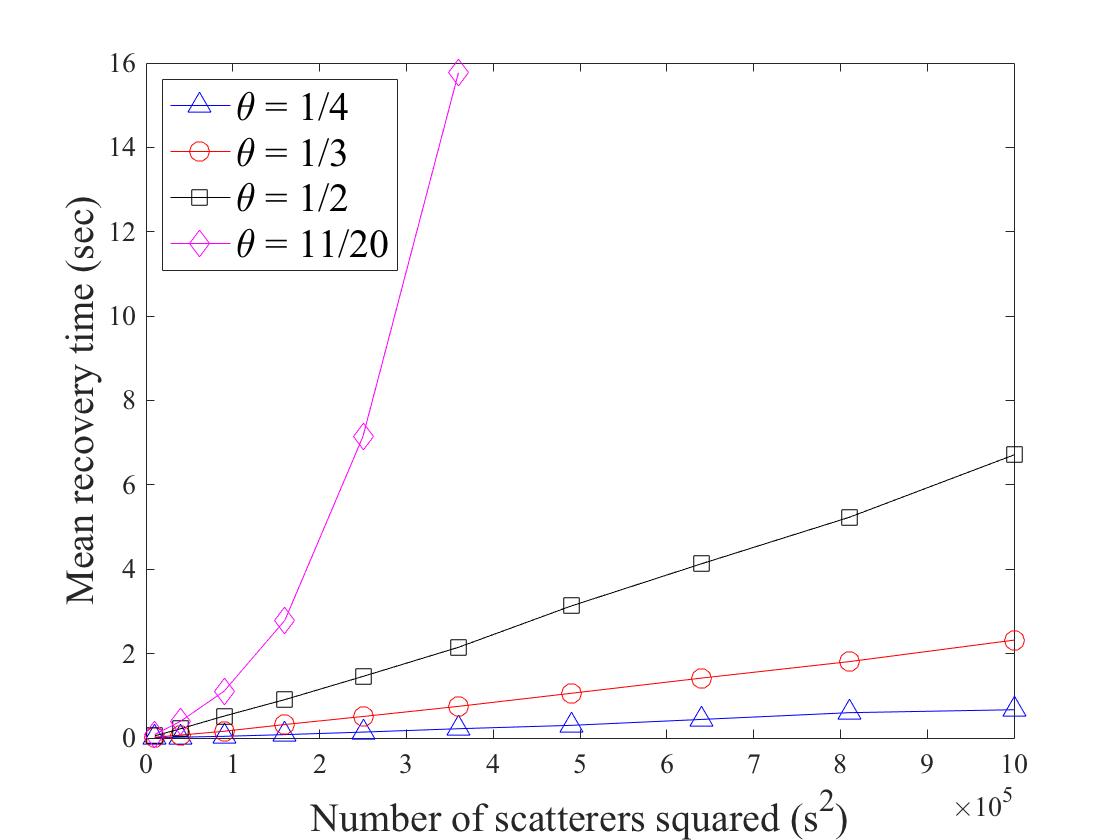}
  \caption{Recovery time v. number of scatterers, uniform distribution}
  \label{ProbScatUnifFig}
\end{figure}

\begin{figure}[h]
  \centering
  \includegraphics[width=.9\linewidth]{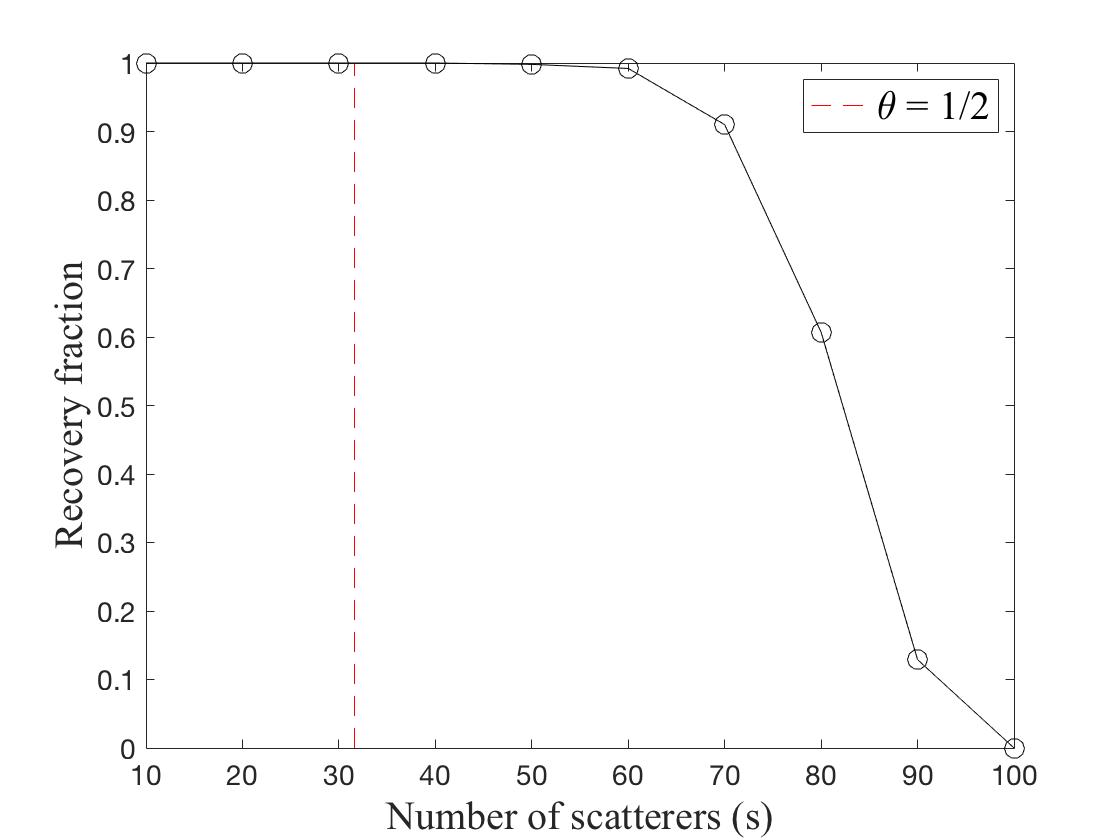}
  \caption{Recovery probability v. number of scatterers, uniform distribution}
\label{TimeScatUnifFIG}
\end{figure}

\begin{enumerate}[wide, labelwidth=!, labelindent=0pt]
\item Our first simulation, displayed in figure \ref{ProbScatUnifFig}, shows the recovery probability of MISTR when $\uptau = 30$, $d=3$ and $n = 10$. As in the uniform case, good recovery performance persists for a time above the theoretical $\theta < 1/2$ threshold for recovery.

\item The simulation documented in figure \ref{TimeScatUnifFIG} shows the time complexity of MISTR in the uniform case when sparsity $s$ grows as a fixed power of $n^d$ for $d = 3$. Here, $\uptau = 30$ while $n = \lfloor s^{1/\ln(\theta)} \rfloor$ for various values of $\theta$ (the integer floor is required since $n$ must be an integer for the uniform distribution). When $\theta \leq 1/2$, the average-case time complexity of MISTR is $\bigO{\uptau s^2\log(s)} \approx \bigO{\uptau k^2\log{k}}$ as before. The data confirms theoretical predictions that MISTR performs well both in terms of recovery probability and runtime. As with the Gaussian case, even when above the recovery threshold, we note that no recovery failures were observed for any of the plotted sparsity levels $\theta$. Though time complexity increases above $\theta = 1/2$, this lies beyond the theoretically-guaranteed threshold for recovery and $\bigO{k^2\log{k}}$ runtime.
\end{enumerate}

\section{Notations}
In table \ref{notations}, we detail important notations for easy reference.

\begin{table*}[b]
\caption{Notations} 
\centering 
\begin{minipage}{0.5\textwidth}
\centering
\begin{tabular}{ |p{0.18\textwidth}|p{0.54\textwidth}|} 
\hline
Notation & Definition \\ [0.5ex] 
\hline 
$d$ & Dimension. Typically greater than 1.\\
$\bfF$ & Fourier transform operator\\ 
$|A|$ & Number of elements in a set A \\
$\bfx$ & Unknown signal defined on $\{0,1,2,\ldots, m-1\}^d$\\
$\supp{\bfx}$ & support of $\bfx$\\
$\autocorr{\bfx}$ & Autocorrelation of $\bfx$, often denoted $\bfa$ \\
$V$ & Unknown set to be recovered\\
$\diff{V}$ & Difference set of $V$: $\{\bfv - \bfv' : \bfv,\bfv' \in V\}$\\
$n$ & Resolution parameter\\
$\mathbb{Z}_n^d$ & $\left(\frac{1}{n}\mathbb{Z}\right)^d$\\
$k$ & Number of elements in a set, usually $V$\\
$W$ & Represents a difference set, usually $\diff{V}$\\
$\kappa$ & Number of elements in a difference set, usually $\diff{V}$\\
$k_{\min}(\kappa)$ & The smallest integer $k$ such that a set with $k$ elements could have a difference set of size $\kappa$ \\
$s$ & Sparsity parameter for probabilistic set model \\
$\theta$ & $s$ is modeled as a power $\theta$ of $n^d$\\
$\uptau$ & Number of random projections employed by MISTR\\
$\bfz^i$ & Independent random vectors drawn uniformly from $\mathbb{S}^{d-1}$, $i \in \{1,2,\ldots,\uptau\}$\\
$\bfv_\eta^i$ & $\eta$-th element of $V$ when ordered according to inner product with $\bfz^i$, ascending\\
$\tV^i$ & Equivalent solution of comb. problem such that $\tV^i \subseteq W$. See eqn. (\ref{truesoln})\\
$\tbfv^i_\eta$ & $\eta$-th element of $\tV^i$ when ordered according to inner product with $\bfz^i$\\
$W^i$ & Elements of $W$ with nonnegative inner product with $\bfz^i$, ordered by inner product with $\bfz^i$\\
$U^i$ & Set recovered from the $i$-th intersection step. $U^i = \{0\}\cup\left[W^i \cap \left(W^i + \tbfv^i_0\right)\right]$\\
$\bfu^i_\eta$ & $\eta$-th element of $U^i$ ordered according to inner product with $\bfz^i$\\
$\bfu^i_{\max}$ & Largest element of $U^i$ ordered according to inner product with $\bfz^i$\\
\hline 
\end{tabular}
\end{minipage}\hfill
\begin{minipage}{0.5\textwidth}
\centering
\begin{tabular}{ |p{0.16\textwidth}|p{0.54\textwidth}|} 
\hline
Notation & Definition \\ [0.5ex] 
\hline 
$\bfl$ & Represents a potential false positive element that persists through the MISTR algorithm\\
$\Theta^i$ & Transformation consisting of possible translation and multiplication by -1, according to the relation $\Theta^i(\tV^i) = \tV^1$ \\
$U$ & Result of the collaboration search: $U = \cap_{i=1}^\uptau \Theta^i\left(U^i\right)$\\
$\bfj$ & A \textit{search sequence} of pairs of the form $\bfj = \{(j_1, \omega_1\}_1^p$ see definition \ref{search}\\
$O^i_\bfj$ & The $i$-th orientation of $\bfj$: a candidate for $\Theta^i$ defined by $O^i_\bfj(\bfu) = \omega_i(\bfu - \bfu_\bfj^i)$ for $\bfu \in U^i$\\
$\mcC_\bfj$ & Collaboration with respect to $\bfj$: $\cap_{i=1}^p O^i_\bfj(U^i)$\\
$p$ & Length of a search sequence or collaboration, called the \textit{depth}\\
$\bfj^*$ & A search sequence of depth $\uptau$ such that $\mcC_\bfj = U$; subsequence up to depth $p$ is denoted $\bfj^*_p$\\
$P_n(\bfc)$ & $\ell^\infty$ ball of radius $\frac{1}{2n}$ centered at $\bfc$, called a \textit{pixel}\\
$\tZnd$ & Set of centers of pixels which intersect the sphere of radius two: $\tZnd = \{\bfc \ in \mathbb{Z}_n^d : P_n(\bfc) \cap B(0,2) \neq \emptyset \}$\\
$\Phi_s$ & Probability measure associated with a $N(0,I/\ln(s))$ variable\\
$\Pi_s$ & Poisson point process on $\mathbb{R}^d$ with underlying measure $\Phi_s$.\\
$V_s$ & Discretized point process on $\mathbb{Z}_n^d: \bfc \in V_s \Leftrightarrow |P_n(\bfc) \cap \Pi_s| \geq 1$\\
$\mfF_\bfl$ & Event that $\{\bfl - \bfv_0^i, \bfl - \bfv_1^i\} \subseteq W^i$ for all $i = 1,2,\ldots \uptau$\\
$V_\uptau$ & Multiset $\{\bfv_0^1,\bfv_1^1,\bfv_0^2,\bfv^1_2,\ldots,\bfv_0^\uptau,\bfv_1^\uptau\}$; can have repeat elements\\
$\dist{A}$ & Set of distinct elements in a multiset $A$\\
$\conv{A}$ & Convex hull of set $A$\\
$\Vertex{A}$ & Set of vertices of the convex hull of set $A$\\
$\Omega_s$ & Event that all points in $\Pi_s$ fall inside $B(0,2)$\\
$\rho_n$ & $1-\exp\left(\frac{-s\ln^{d/2}(s)}{n^d}\right) \approx 1 - \frac{s\ln^{d/2}(s)}{n^d}$\vspace{13pt}\\
\hline
\end{tabular}
\end{minipage}
\label{notations}
\end{table*}

\ifCLASSOPTIONcaptionsoff
  \newpage
\fi

 \end{document}